\documentclass{amsart}
\RequirePackage[utf8]{inputenc}
\usepackage{amsmath,amsthm,amsfonts,amssymb,amscd,amsbsy,multirow,hyperref}
\usepackage[all]{xy}
\usepackage{graphicx,pgf,caption,subcaption}
\usepackage{enumerate}
\usepackage{mathdots}
\usepackage{dsfont}
\usepackage[linesnumbered,ruled,vlined]{algorithm2e}
\usepackage{algpseudocode}

\newcommand{\conv}{\operatorname{conv}}
\newcommand{\cone}{\operatorname{cone}}

\newcommand{\disc}{\operatorname{disc}}
\newcommand{\dd}{\mathrm{d}}
\newcommand{\id}{\operatorname{Id}}

\newcommand{\Z}{\mathds Z}

\newcommand{\R}{\mathds R}
\newcommand{\C}{\mathds C}

\newcommand{\SO}{\mathsf{SO}}
\renewcommand{\O}{\mathsf O}

\newcommand{\GL}{\mathsf{GL}}

\newcommand{\Sym}{\operatorname{Sym}}

\newcommand{\g}{\mathrm g}
\newcommand{\tr}{\operatorname{tr}}
\newcommand{\Gr}{\operatorname{Gr}_2}

\newcommand{\diag}{\operatorname{diag}}

\allowdisplaybreaks

\hypersetup{
    pdftoolbar=true,
    pdfmenubar=true,
    pdffitwindow=false,
    pdfstartview={FitH},
    pdftitle={Convex Algebraic Geometry of Curvature Operators},
    pdfauthor={Renato G. Bettiol, Mario Kummer, and Ricardo A. E. Mendes},
    pdfsubject={},
    pdfkeywords={}
    pdfnewwindow=true,
    colorlinks=true, 
    linkcolor=blue,
    citecolor=blue,
    urlcolor=black,
}

\newtheorem{theorem}{Theorem}[]
\newtheorem{lemma}[theorem]{Lemma}
\newtheorem*{claim}{Claim}
\newtheorem{proposition}[theorem]{Proposition}
\newtheorem{corollary}[theorem]{Corollary}

\newtheorem{mainthm}{\sc Theorem}

\newtheorem*{mainthmp}{\sc Theorem A'}

\theoremstyle{definition}
\newtheorem{definition}[theorem]{Definition}

\theoremstyle{remark}
\newtheorem{remark}[theorem]{Remark}
\newtheorem{remarks}[theorem]{Remarks}
\newtheorem{example}[theorem]{Example}

\title[Convex Algebraic Geometry of Curvature Operators]{Convex Algebraic Geometry \\ of Curvature Operators}

\author[R. G. Bettiol]{Renato G. Bettiol}
\address{City University of New York (Lehman College) \newline
\indent Department of Mathematics  \newline
\indent 250 Bedford Park Blvd W\newline
\indent Bronx, NY, 10468, USA }
\email{r.bettiol@lehman.cuny.edu}

\author[M. Kummer]{Mario Kummer}
\address{Technische Universit\"at Dresden \newline
\indent Institut f\"ur Geometrie \newline
\indent  Willersbau, WIL B 109 \newline
\indent  Zellescher Weg 12-14 \newline
\indent 01069 Dresden, Germany}
\email{mario.kummer@tu-dresden.de}

\author[R. A. E. Mendes]{Ricardo A. E. Mendes}
\address{University of Oklahoma\newline
\indent Department of Mathematics\newline
\indent 601 Elm Ave\newline
\indent Norman, OK, 73019-3103, USA}
\email{ricardo.mendes@ou.edu}

\allowdisplaybreaks
\numberwithin{equation}{section}
\numberwithin{theorem}{section}


\date{\today}

\begin{document}
\begin{abstract}
We study the structure of the set of algebraic curvature operators satisfying a sectional curvature bound under the light of the emerging field of Convex Algebraic Geometry. More precisely, we determine in which dimensions $n$ this convex semialgebraic set is a spectrahedron or a spectrahedral shadow; in particular, for $n\geq5$, these give new counter-examples to the Helton--Nie Conjecture. Moreover, efficient algorithms are provided if $n=4$ to test membership in such a set. For $n\geq5$, algorithms using semidefinite programming are obtained from hierarchies of inner approximations by spectrahedral shadows and outer relaxations by spectrahedra.
\end{abstract}

\maketitle

\section{Introduction}

The emerging field of Convex Algebraic Geometry originates from a natural coalescence of ideas in Convex Geometry, Optimization, and Algebraic Geometry, and has witnessed great progress over the last few years, see \cite{SIAMbook} for surveys. The main objects considered are convex semialgebraic subsets of vector spaces, such as \emph{spectrahedra} and their \emph{shadows}; and their study has led to remarkable achievements in optimization problems for polynomials in several variables.
In particular, \emph{semidefinite programming} on spectrahedral shadows is a far-reaching generalization of linear programming on convex polyhedra, and an area of growing interest due to its numerous and powerful applications, see e.g. \cite{handbook}.

The \emph{raison d'\^etre} of this paper is to shed new light on curvature operators of Riemannian manifolds with sectional curvature bounds from the viewpoint of Convex Algebraic Geometry. More importantly, we hope that the connections established here will serve as foundations for developing further ties between the exciting new frontiers conquered by Convex Algebraic Geometry and classical objects and open problems from Geometric Analysis and Riemannian Geometry.

Recall that a \emph{semialgebraic set} is a subset $S\subset \R^n$ defined by boolean combinations of finitely many polynomial equalities and inequalities; for example, the set $S\subset\R^4$ consisting of $(a,b,c,x)\in\R^4$ such that $ax^2+bx+c=0$ and $a\neq0$ is a semialgebraic set. 
The celebrated Tarski--Seidenberg Theorem states that linear projections of semialgebraic sets are also semialgebraic. As an illustration, consider the image $\pi(S)\subset\R^3$ of $S\subset\R^4$ under the projection $\pi(a,b,c,x)=(a,b,c)$. It consists precisely of $(a,b,c)\in\R^3$ with $a\neq0$ for which
\begin{equation}\label{eq:quantified}
\exists\, x\in\R \quad \text{such that} \quad ax^2+bx+c=0,
\end{equation}
and it can also be described by finitely many polynomial equalities and inequalities (without quantifiers), namely:
\begin{equation}\label{eq:unquantified}
b^2-4ac\geq0.
\end{equation}
The algorithmic process of rewriting a quantified polynomial sentence, such as \eqref{eq:quantified}, as an equivalent polynomial sentence without quantifiers, such as \eqref{eq:unquantified}, is known as \emph{Quantifier Elimination}. This method generalizes the Tarski--Seidenberg Theorem as formulated above, and has deep consequences in Logic, Model Theory, and Theoretical Computer Science.

An elementary application of Quantifier Elimination to Riemannian Geometry is to eliminate the quantifier $\forall$ from the sentence that defines a sectional curvature bound. 
For example, the condition $\sec\geq k$ for 
an algebraic curvature operator $R\colon\wedge^2\R^n\to\wedge^2\R^n$, 
is given by the (quantified) sentence
\begin{equation*}
\forall \,\sigma\in\Gr^+(\R^n),\quad  \sec_R(\sigma):=\langle R(\sigma),\sigma\rangle \geq k,
\end{equation*}
where $\Gr^+(\R^n)=\{X\wedge Y\in\wedge^2\R^n:\|X\wedge Y\|=1\}$ is the (oriented) Grassmannian of $2$-planes in~$\R^n$, which is a real algebraic variety hence also a semialgebraic set.
This had been observed, among others, by Weinstein~\cite[p.~260]{Weinstein72}:
\begin{quote}\em
there exist finitely many polynomial inequalities in the $R_{ijkl}$'s such that, given any curvature tensor, one could determine whether it is positive sectional by evaluating the polynomials and checking whether the results satisfy the inequalities.
\end{quote}
In other words, the sets
\begin{equation*}
\mathfrak R_{\sec\geq k}(n):=\big\{R\in\Sym^2_b(\wedge^2\R^n) : \sec_R\geq k\big\}
\end{equation*}
are semialgebraic. Here, the subscript $_b$ indicates that $R\in\Sym^2(\wedge^2\R^n)$ satisfies the first Bianchi identity, see Section~\ref{sec:cagprel} for preliminaries on Riemannian Geometry and curvature operators. The first Bianchi identity is stated in \eqref{eq:bianchi}. We stress that $\mathfrak R_{\sec\geq 0}(n)$ can be thought of as the subset of all forms of degree two in the homogeneous coordinate ring of the Grassmannian $\Gr(n)$ that are nonnegative on the real part of $\Gr(n)$. This point of view might be helpful for readers with a background in Convex Algebraic Geometry.
Weinstein~\cite[p.~260]{Weinstein72} continues:
\begin{quote}\em
It would be useful to know these inequalities explicitly.~[...]
Unfortunately, the [Quantifier Elimination] procedure is too long to be used in practice even with the aid of a computer.
\end{quote}
Despite all technological advances, this remains true today, almost 50 years later. Although such an explicit description of $\mathfrak R_{\sec\geq k}(n)$ is still elusive, in this paper we provide new information about these semialgebraic sets. Besides being of intrinsic interest, we expect this will lead to new global results in differential geometry.

A fundamental example of convex semialgebraic set is the cone $\{A\in\Sym^2(\R^d):A\succeq0\}$ of positive-semidefinite matrices.
Preimages of this cone under affine maps $\R^n\to \Sym^2(\R^d)$ 
are also convex semialgebraic, and 
are called \emph{spectrahedra}.
They generalize \emph{polyhedra}, which correspond to affine maps with image in the subspace of diagonal matrices. In contrast to polyhedra, the linear projection of a spectrahedron may fail to be a spectrahedron. Nevertheless, these projections are convex semialgebraic sets, and are called \emph{spectrahedral shadows}.
Following a question of Nemirovski~\cite{nemirovski} in his 2006 ICM plenary address, Helton and Nie~\cite[p.~790]{helton-nie} conjectured that \emph{every} convex semialgebraic set is a spectrahedral shadow. Remarkably, this turned out not to be the case, as very recently discovered by Scheiderer~\cite{claus}. 
Further counter-examples were subsequently found in~\cite{hamza}. We describe a convenient and digestible criterion, which follows from Scheiderer's work~\cite{claus}, for the cone of nonnegative polynomials inside a given vector space of polynomials not to be a spectrahedral shadow that can be of independent interest.

Our first main result describes how sets of algebraic curvature operators with sectional curvature bounds fit in the above taxonomy of convex semialgebraic sets, providing yet another class of counter-examples to the Helton--Nie Conjecture:

\begin{mainthm}\label{mainthm:spectrahedra}
For all $k\in\R$, each of the sets $\mathfrak R_{\sec\geq k}(n)$ and $\mathfrak R_{\sec\leq k}(n)$ is:
\begin{enumerate}[\indent \rm (1)]
\item\label{A3} not a spectrahedral shadow, if $n\geq5$;
\item\label{A2} a spectrahedral shadow, but not a spectrahedron, if $n=4$;
\item\label{A1} a spectrahedron, if $n\leq3$.
\end{enumerate}
\end{mainthm}

We state Theorem~\ref{mainthm:spectrahedra} in the above manner for the sake of completeness, despite the fact that some claims were previously known.
More precisely, statement~\eqref{A1} follows trivially from the equivalence, in dimensions $n\leq3$, between $\sec_R\geq k$ and $R-k\id\succeq0$; analogously for $\sec\leq k$ (which we omit henceforth, see Remark~\ref{rem:obvious}). Furthermore, the first part of statement \eqref{A2} follows from the so-called \emph{Thorpe's trick}~\cite{Thorpe72}, see Proposition~\ref{prop:thorpe}; namely, the equivalence, in dimension $n=4$, between $\sec_R\geq k$ and the existence of $x\in\R$ such that $R-k\id+x\,*\succeq0$, where $*$ is the Hodge star operator. In other words, $\mathfrak R_{\sec\geq k}(4)$ is the image of the spectrahedron $\{R\in\Sym^2(\wedge^2\R^4):R-k\id\succeq0\}$ under the orthogonal projection $\Sym^2(\wedge^2\R^4)\to\Sym^2_b(\wedge^2\R^4)$, whose kernel is spanned by~$*$.
We point out that Thorpe's trick is actually an instance of a much older result due to Finsler~\cite{finsler}, see Lemma~\ref{lemma:finsler}; a fact that seems to have gone unnoticed thus far. 

For readers interested in \emph{strict} sectional curvature bounds $\sec>k$ and $\sec<k$, we note that statements \eqref{A3} and \eqref{A2} in Theorem~\ref{mainthm:spectrahedra} 
carry over to this setting, see Remarks~\ref{rem:obvious} and \ref{rem:specrem}.
However, in keeping with the Convex Algebraic Geometry literature, all spectrahedra are (by definition) closed sets. Thus, $\mathfrak R_{\sec>k}(n)$ and $\mathfrak R_{\sec<k}(n)$, $n\leq 3$, are, strictly speaking, not spectrahedra.
Of course, this is just a matter of convention, and $\sec>k$ is clearly equivalent to $R-k\id\succ0$ if $n\leq3$. 

Although $\mathfrak R_{\sec\geq k}(n)$, $n\geq5$, fails to be a spectrahedral shadow, our second main result provides natural approximations by spectrahedral shadows and spectrahedra:

\begin{mainthm}\label{mainthm:approx}
For all $k\in\R$ and $n\geq 2$, there are inner and outer approximations of $\mathfrak R_{\sec\geq k}(n)$ by nested sequences $\mathfrak I_m$ of spectrahedral shadows and $\mathfrak O_m$ of spectrahedra,
\begin{equation*}
\mathfrak I_0\, \subset\,  \mathfrak I_1\, \subset \dots \subset \, \mathfrak I_m\, \subset \dots\subset   \mathfrak R_{\sec\geq k}(n)  \subset \dots\subset \, \mathfrak O_m\, \subset\dots\subset\,  \mathfrak O_1\, \subset\, \mathfrak O_0, \quad m\geq 0,
\end{equation*}
which are $\O(n)$-invariant and satisfy
$\overline{\bigcup_{m\geq0}\mathfrak I_m}=\mathfrak R_{\sec\geq k}(n)=\bigcap_{m\geq0}\mathfrak O_m$. 
\end{mainthm}

The inner approximation by spectrahedral shadows $\mathfrak I_m$, $m\geq0$, is constructed in the same spirit as the \emph{Lasserre hierarchy}~\cite{lasserre, parrilothesis}; and, if $k=0$, the first step $\mathfrak I_0$ coincides precisely with the subset of curvature operators with \emph{strongly nonnegative curvature}, see~\cite{strongnonneg,strongpos}.
The outer approximation by spectrahedra $\mathfrak O_m$, $m\geq0$, is given by curvature operators with positive-semidefinite  curvature terms in all Weitzenb\"ock formulae for traceless symmetric $p$-tensors with $p\leq m+1$, see~\cite[Thm.~A]{weitzenbock}; and, if $k=0$, the first step $\mathfrak O_0$ coincides with the subset of curvature operators with nonnegative Ricci curvature. Since all $\mathfrak I_m$ and $\mathfrak O_m$ are $\O(n)$-invariant, these approximations are \emph{geometric}, in the sense that they define coordinate-free curvature conditions.
We remark that $\mathfrak I_0=\mathfrak R_{\sec\geq k}(n)$ if and only if $n\leq 4$, while $\mathfrak O_0=\mathfrak R_{\sec\geq k}(n)$ if and only if $n=2$. By Theorem~\ref{mainthm:spectrahedra}, these approximations  \emph{do not stabilize} after finitely many steps $m\geq0$, for all $n\geq5$.

In our third main result, we restrict to dimension $n=4$ to exploit the description of $\mathfrak R_{\sec\geq k}(4)$ as a spectrahedral shadow in order to obtain an \emph{explicit description} of this set as an \emph{algebraic interior}, see Definition~\ref{def:alginterior}.

\begin{mainthm}\label{mainthm:4dim}
The set $\mathfrak R_{\sec\geq k}(4)$ is an algebraic interior with minimal defining polynomial $\mathfrak p_k\colon\Sym^2_b(\wedge^2\R^4)\to\R$, given by
\begin{equation}\label{eq:pk}
\mathfrak p_k(R)=\disc_x\!\big(\!\det(R-k\id +x\,*)\big).
\end{equation}
More precisely, $\mathfrak R_{\sec\geq k}(4)=\overline{\mathcal C_k}$, where $\mathcal C_k$ is the only connected component of the set $\big\{R\in \Sym^2_b(\wedge^2\R^4):\mathfrak p_k(R)>0\big\}$ such that $(k+1)\id\in\overline{\mathcal C_k}$.
\end{mainthm}

In the above, $\disc_x$ denotes the discriminant in $x$, see Subsection~\ref{subsec:disc}.
Using a \emph{complexification trick} \eqref{eq:Ttrick}, the above polynomial $\mathfrak p_k(R)$ can be seen as the discriminant of a symmetric matrix, and hence more efficiently computed, e.g.~ using~\cite{parlett}. A result related to Theorem~\ref{mainthm:4dim}, where \eqref{eq:pk} is considered as a polynomial in $k\in\R$, was recently obtained by Fodor~\cite{fodor}.

Although Theorem~\ref{mainthm:4dim} falls short of giving a description of $\mathfrak R_{\sec\geq k}(4)$ as a semialgebraic set, it provides an explicit such description of another semialgebraic set that has $\mathfrak R_{\sec\geq k}(4)$ as (the closure of) one of its connected components.
Moreover, it follows from Theorem~\ref{mainthm:4dim}
that the algebraic boundary of $\mathfrak R_{\sec\geq k}(4)$, i.e., the Zariski closure of its topological boundary, is the zero set of the polynomial \eqref{eq:pk}.

As a computational application of the description of $\mathfrak R_{\sec\geq k}(4)$ as a spectrahedral shadow, we provide an efficient algorithm (different from semidefinite programming) to determine when a given $R\in\Sym^2_b(\wedge^2\R^4)$ belongs to this set.
This algorithm is based on \emph{Sturm's root counting} method, and detects membership in $\mathfrak R_{\sec\geq 0}(4)$ and also in $\mathfrak R_{\sec>0}(4)$, see Algorithms \ref{alg:poscurv} and \ref{alg:nonnegcurv}; the cases of other sectional curvature bounds (strict or not) are easily obtained from these with obvious modifications.
For $n\geq5$,  the approximations given by Theorem~\ref{mainthm:approx} allow to use
an iteration of semidefinite programs (see Algorithm~\ref{alg:relax}) to detect
membership in $\mathfrak R_{\sec\geq k}(n)$
except for a set of measure zero of bad inputs, 
where the algorithm does not halt.

Straightforward generalizations of Theorems~\ref{mainthm:spectrahedra}, \ref{mainthm:approx}, and \ref{mainthm:4dim} to 
sets of algebraic curvature operators on vectors spaces with \emph{semi-Riemannian inner products}, satisfying a suitable replacement for $\sec\geq k$, are discussed in Appendix~\ref{sec:semiriem}.

\subsection*{Relevance for applications}
Next, we highlight some applications to which this work is pertinent in Optimization, Information Theory, and Data Science. 

\subsubsection*{Semidefinite programming}
Semidefinite programming is a generalization of linear programming which attracted a lot of interest, as one can solve under mild assumptions a semidefinite program up to a fixed precision in time that is polynomial in the program description size \cite{nesterov}. Typical applications include polynomial optimization \cite{parrilo-sturmfels} or combinatorial optimization, e.g.~the Max-Cut Problem \cite{maxcut}. The standard form for a semidefinite program (SDP) is:
\begin{equation*}
\min_x \,\, c_1 x_1 + \ldots + c_n x_n ,
\end{equation*}
under the constraint that the real symmetric matrix $A_0+x_1 A_1 + \ldots + x_n A_n$ is positive-semidefinite, where $c_i \in \R$ and $A_0, \ldots , A_n$ are real symmetric matrices. Thus the feasible region of an SDP is a spectrahedron. By introducing slack variables one can also run an SDP on spectrahedral shadows. Therefore, in the theory of semidefinite programming, it is crucial to find a characterization of this special class of convex sets \cite{nemirovski}. While it is easy to see that every spectrahedral shadow is convex and semialgebraic, the first examples of convex semialgebraic sets that are not spectrahedral shadows were only found recently by Scheiderer~\cite{claus}. In this work, we provide further such examples (Theorem~\ref{mainthm:spectrahedra}). Moreover, we believe that our simplification and concretization of Scheiderer's criterion for being a spectrahedral shadow (Theorem~\ref{thm:claus}) will be useful in the future for assessing the applicability of SDP. The question on the expressive power of semidefinite programming is also of ample interest in Theoretical Computer Science because of its intimate connection to the Unique Games Conjecture~\cite{ugc}.

\subsubsection*{Information Geometry}
Consider the space of parameters $M$ of a probability distribution $p(x,\theta)$ with the Riemannian metric given by 
the Fisher information matrix
\begin{equation}\label{eq:fisher}
\g_{ab}(\theta)=\mathrm E\left[\frac{\partial^2 i(x,\theta)}{\partial \theta_a\partial \theta_b} \right]=
\int_X \frac{\partial^2 i(x,\theta)}{\partial \theta_a\partial \theta_b}\,p(x,\theta)\;\dd x,
\end{equation}
where $\theta=(\theta_1,\dots,\theta_n)\in M$ are local coordinates, $i(x,\theta)=-\log p(x,\theta)$ is the Shannon information, and $x$ is drawn from the value space of the random variable~$X$.
Such Riemannian manifolds $(M,\g)$ are called \emph{statistical manifolds}, and are the central objects of study in Information Geometry, which leverages methods from Differential Geometry and Geometric Analysis in Probability Theory, Statistical Inference, and Information Science. Applications of Information Geometry are fast growing and widespread, ranging from neural networks, machine learning and signal processing~\cite{amari} to
complexity of composite systems, evolutionary dynamics, and Monte Carlo stochastic sampling~\cite{ig-book}.
Since the Riemannian curvature of the Fisher metric \eqref{eq:fisher} allows to detect critical parameter values where a phase transition occurs~\cite[p.~100]{ig-book}, our Algorithms~\ref{alg:relax}, \ref{alg:poscurv}, and \ref{alg:nonnegcurv}, in Section~\ref{sec:relaxalgo}, may provide valuable computational tools for such applications. Namely, given a probability distribution with $n$ parameters, the full range of sectional curvatures at a point $\theta$ in the corresponding $n$-dimensional statistical manifold $(M,\g)$ can be determined if $n\leq4$, or approximated if $n\geq5$, by repeatedly invoking the above algorithms with the curvature operator of $(M,\g)$ at $\theta$ as input, and raising/lowering the corresponding lower/upper curvature bounds.

\subsubsection*{Geometric Data Analysis}
Many tools in Data Science, such as the notion of \emph{mean} (or \emph{centroid}), and its use, e.g., in $k$-means clustering, and Principal Component Analysis (PCA), which is fundamental to address the ``curse of dimensionality'', were originally conceived for data lying inside Euclidean space $\R^n$. Given the recent considerable availability and interest in manifold-valued data, the development of suitable nonlinear replacements for these tools became of great importance.
Fr\'echet mean (FM) and Principal Geodesic Analysis (PGA), respectively, are convenient generalizations that 
attracted particular interest for applications in medical imaging~\cite{pga0,pga}. 
However, these can only be used assuming that the (known) manifold $(M,\g)$, where the data points lie, satisfies certain geometric constraints. For instance, uniqueness of FM for data points in $(M,\g)$ with $\sec\leq K$, $K>0$, is guaranteed if these points lie inside a geodesic ball of radius $r<\tfrac12\min\!\left\{\operatorname{inj} M, \frac{\pi}{\sqrt{K}} \right\},$ where $\operatorname{inj} M$ is the injectivity radius of $(M,\g)$, see~\cite{afsari}. Similarly, PGA relies on closest-point projection operators onto geodesic submanifolds of $(M,\g)$, which can also be controlled in terms of sectional curvature bounds. Thus, Algorithms~\ref{alg:relax}, \ref{alg:poscurv}, and \ref{alg:nonnegcurv}, in Section~\ref{sec:relaxalgo}, through their estimation of sectional curvature bounds, allow to estimate on which (portions of) manifolds FM and PGA can be applied reliably for data analysis. Note that these algorithms only determine sectional curvature bounds \emph{point-by-point}, so the above strategy involves using a sufficiently dense net of sample points on $(M,\g)$ to which such algorithm is applied.

\subsection*{Discussion of proofs}
In order to simplify the exposition, now and throughout the paper we only consider the sectional curvature bound $\sec\geq0$, for the reasons laid out in Remark~\ref{rem:obvious}. The proof of Theorem~\ref{mainthm:spectrahedra}~\eqref{A3} uses much heavier theoretical machinery than Theorem~\ref{mainthm:4dim}, namely the deep recent results of Scheiderer~\cite{claus}, while  Theorem~\ref{mainthm:spectrahedra}~\eqref{A2} is a consequence of Theorem~\ref{mainthm:4dim}. Theorem~\ref{mainthm:approx} relies on \cite{weitzenbock} to produce the outer approximation by spectrahedra, and on an adaptation of the Lasserre hierarchy method for the inner approximation by spectrahedral shadows.

There are three main steps in the proof of Theorem~\ref{mainthm:4dim}, which is presented in Section~\ref{sec:dim4}. First, a complexification trick \eqref{eq:Ttrick} is used to establish that $\mathfrak p:=\mathfrak p_0$, see \eqref{eq:pk}, vanishes on the topological boundary of $\mathfrak R_{\sec\geq0}(4)$, see Proposition~\ref{prop:zeroboundary}. Second, 
we show that the vanishing locus of $\mathfrak p$ does not disconnect the  interior of $\mathfrak R_{\sec\geq0}(4)$ for dimensional reasons, see Proposition~\ref{prop:codiminterior}. These two facts already imply that $\mathfrak R_{\sec\geq0}(4)$ is an algebraic interior with defining polynomial $\mathfrak p$, so it only remains to prove that $\mathfrak p$ is \emph{minimal}. This 
follows from irreducibility of $\mathfrak p$, which is an application of canonical forms for complex symmetric matrices, see  Appendix~\ref{appendix}.

The intimate connection between algebraic curvature operators and quadratic forms on the Grassmannian of $2$-planes is at the foundation of the proof of Theorem~\ref{mainthm:spectrahedra}~\eqref{A3}.
More generally, given a real projective variety $X\subset\C P^N$, and a (real) quadratic form $f$ on $X$, that is, an element in the degree 2 part $\R[X]_2$ of its homogeneous coordinate ring, the \emph{value} of $f$ at a real point $x\in X(\R)$ is not well-defined, however its \emph{sign} is. Indeed, $f(\lambda\tilde x)=\lambda^2 f(\tilde x)$ for any representative $\tilde x\in\R^{N+1}$ and $\lambda\in\R\setminus\{0\}$.
Thus, one may consider the set $\textrm P_X\subset\R[X]_2$ of all nonnegative quadratic forms on $X(\R)$, see \eqref{eq:PX}, which clearly contains the set $\Sigma_X$ of quadratic forms that are sums of squares of elements in $\R[X]_1$, see \eqref{eq:SigmaX}.
The characterization of varieties $X$ for which $\textrm P_X=\Sigma_X$ as those of \emph{minimal degree} is a landmark result recently obtained by Blekherman, Smith, and Velasco~\cite{blekhermanetal}.
This fits into the broader question of which nonnegative functions are sums of squares, which has a long history, dating back to Minkowski, Hilbert~\cite{hilbert} and Artin~\cite{artin}.

In the case of the Grassmannian $X=\Gr(n)$, 
which is determined by quadratic equations $\omega_i=0$ called \emph{Pl\"ucker relations}, 
see \eqref{eq:grassm}, 
the connection alluded to above takes the form of
the identification
\begin{equation}\label{eq:ident}
\R[\Gr(n)]_2\cong\Sym^2_b(\wedge^2\R^n).
\end{equation}
Namely, $\Sym^2(\wedge^2\R^n)$ can be identified as usual with quadratic forms on $\wedge^2\R^n$ by associating each 
$R$ to $q_R(\alpha)=\langle R(\alpha),\alpha\rangle$. 
On the one hand, $\R[\Gr(n)]_2$ is by definition the \emph{quotient} $\Sym^2(\wedge^2\R^n)/\operatorname{span}(\omega_i)$, since $\omega_i$ generate the vanishing ideal of $\Gr(n)$. On the other hand, the \emph{orthogonal complement} of $\operatorname{span}(\omega_i)$ in $\Sym^2(\wedge^2\R^n)$ is exactly $\Sym^2_b(\wedge^2\R^n)$, yielding \eqref{eq:ident}, see Subsection~\ref{subsec:grasscurvop} for details.

Under the identification \eqref{eq:ident}, the set $\textrm P_{\Gr(n)}\subset\R[\Gr(n)]_2$ corresponds to $\mathfrak R_{\sec\geq0}(n)\subset\Sym^2_b(\wedge^2\R^n)$, while $\Sigma_{\Gr(n)}$ corresponds to the set of curvature operators with strongly nonnegative curvature, see Example~\ref{ex:grasspsdsos}.
As the Grassmannian~$\Gr(4)$ has minimal degree, $\textrm P_{\Gr(4)}=\Sigma_{\Gr(4)}$ by \cite{blekhermanetal}.
This recovers the Finsler--Thorpe trick, since the only Pl\"ucker relation in dimension $n=4$ is given by $\omega_1(R)=\langle *R,R\rangle=0$. Furthermore, as $\Gr(n)$ does not have minimal degree for all $n\geq5$, there exist $P\in \textrm P_{\Gr(n)}\setminus\Sigma_{\Gr(n)}$ which translates to the
failure of higher-dimensional analogues of the Finsler--Thorpe trick (explicit $P$'s were obtained by Zoltek~\cite{Zoltek79}). 
Such a $P$ is the key input to apply a criterion of Scheiderer~\cite{claus} to show that $\textrm P_{\Gr(n)}\cong \mathfrak R_{\sec\geq0}(n)$ is not a spectrahedral shadow, as claimed in Theorem~\ref{mainthm:spectrahedra}~\eqref{A3}. 
In fact, we extract from \cite{claus} an easily applicable criterion, Theorem~\ref{thm:claus}, which implies \cite[Thm.~3]{hamza} and is of independent interest.

More generally, the Grassmannians $\textrm Gr_k(n)$ of $k$-planes do not have minimal degree 
if and only if $2\leq k\leq n-2$ and $n\geq5$,
and hence $\textrm P_{\textrm Gr_k(n)}\neq\Sigma_{\textrm Gr_k(n)}$ in this range. Scheiderer's criterion still applies in this situation, and leads to the conclusion that $\textrm P_{\textrm Gr_k(n)}$, $2\leq k\leq n-2$, $n\geq5$, are not spectrahedral shadows, see Corollary~\ref{cor:notspec}.
Since $\Sigma_X$ is a spectrahedral shadow for any projective variety $X$, this can be interpreted as a strengthening of \cite{blekhermanetal} for the class of Grassmannians $X=\textrm{Gr}_k(n)\subset\C P^{\binom{n}{k}-1}$; namely,
$\textrm P_X$ is a spectrahedral shadow if and \emph{only if} $\textrm P_X=\Sigma_X$.
The same strengthening was observed by Scheiderer~\cite[Cor.~4.25]{claus} for the class of degree $d$ Veronese embeddings $X=\C P^n\subset \C P^{\binom{n + d}{d}-1}$. Nevertheless, such a strengthening does not hold in full generality, as exemplified by curves $X$ of positive genus (e.g., elliptic curves). Such $X$ do not have embeddings of minimal degree, however $\textrm P_X$ is always a spectrahedral shadow since it is the dual of the convex hull of a curve, which is a spectrahedral shadow by Scheiderer~\cite{scheider1}.

\vspace{-0.02cm}

\subsection*{Acknowledgements} It is a great pleasure to thank
Bernd Sturmfels for introducing the second-named author to the first- and third-named authors, and suggesting the complexification trick~\eqref{eq:Ttrick}. We also thank Alexander Lytchak for supporting a visit by the second-named author to the University of Cologne, during which excellent working conditions allowed us to finalize this paper. The first-named author is grateful to the Max Planck Institut f\"ur Mathematik in Bonn for the hospitality in the summer of 2019, which made it possible to regularly meet the other authors.

The first-named author was supported by the National Science Foundation grant DMS-1904342, and by the Max Planck Institut f\"ur Mathematik in Bonn, and the third-named author was supported by the National Science Foundation grant DMS-2005373 and the Deutsche Forschungsgemeinschaft grants DFG ME 4801/1-1 and DFG SFB TRR 191.

\section{Preliminaries}\label{sec:cagprel}

\subsection{Riemannian Geometry}
Given a Riemannian manifold $(M,\g)$ and $p\in M$, the \emph{curvature operator} at $p$ is the symmetric endomorphism $R\in\Sym^2(\wedge^2 T_pM)$,
\begin{equation}\label{eq:curvopmanifold}
\langle R(X\wedge Y),Z\wedge W\rangle=\g\big(\nabla_Y\nabla_X Z-\nabla_X\nabla_Y Z+\nabla_{[X,Y]}Z,W\big),
\end{equation}
where $\nabla$ denotes the Levi-Civita connection, and $\langle\cdot,\cdot\rangle$ denotes the inner product induced by $\g$ on $\wedge^2 T_pM$. Curvature operators  $R$ satisfy the \emph{(first) Bianchi identity}:
\begin{equation}\label{eq:bianchi}
\langle R(X\wedge Y),Z\wedge W\rangle+\langle R(Y\wedge Z),X\wedge W\rangle+\langle R(Z\wedge X),Y\wedge W\rangle=0,
\end{equation}
for all $X,Y,Z,W\in T_pM$.
Given $X,Y\in T_pM$ two $\g$-orthonormal tangent vectors, the \emph{sectional curvature} of the plane $\sigma$ spanned by $X$ and $Y$ is 
\begin{equation}\label{eq:sec}
\sec(\sigma)=\langle R(X\wedge Y),X\wedge Y\rangle.
\end{equation}

Since the present paper is only concerned with \emph{pointwise} properties of curvature operators, we henceforth identify $T_pM\cong\R^n$ and define \emph{(algebraic) curvature operators} as elements $R\in\Sym^2(\wedge^2\R^n)$ that satisfy the Bianchi identity \eqref{eq:bianchi}. We denote by $\Sym^2_b(\wedge^2\R^n)\subset \Sym^2(\wedge^2\R^n)$ the subspace of such curvature operators, 
and by $\mathfrak b$ the orthogonal projection onto its complement, so that $\Sym^2_b(\wedge^2\R^n)=\ker\mathfrak b$. Elements $R\in\Sym^2(\wedge^2\R^n)$ are sometimes called \emph{modified curvature operators}.

We identify $\wedge^4\R^n$ with a subspace of $\Sym^2(\wedge^2\R^n)$ via
\begin{equation}\label{eq:wedge4sym2wedge2}
\langle \omega(\alpha),\beta\rangle=\langle \omega,\alpha\wedge\beta\rangle, \quad \mbox{ for all } \omega\in\wedge^4\R^n,\; \alpha,\beta\in\wedge^2\R^n.
\end{equation}
Note that $\wedge^4\R^n$ is the  image of $\mathfrak b$, i.e., the orthogonal complement of $\Sym^2_b(\wedge^2\R^n)$.

\subsection{Grassmannians and curvature operators}\label{subsec:grasscurvop}
The above classical definitions from Riemannian geometry can be conveniently reinterpreted in terms of the algebraic geometry of the Grassmannian of $2$-planes.
This relationship forms the \emph{raison d'\^etre} of this paper.

The natural coordinates $x_{ij}$, $1\leq i<j\leq n$, in $\wedge^2\C^n$ induced from the standard basis $e_1,\dots,e_n$ of $\C^n$ are called \emph{Pl\"ucker coordinates}.
The \emph{Grassmannian}
\begin{equation}\label{eq:grassm}
\Gr(n)\subset\mathds{P}(\wedge^2 \C^n)\cong\C P^{\binom{n}{2}-1}
\end{equation}
of $2$-planes in $\C^n$ is the real projective variety defined by the \emph{Pl\"ucker relations}; namely the zero locus of the quadratic forms associated to a basis of $\wedge^4 \R^n$, considered as $\binom{n}{4}$ homogeneous quadratic polynomials on (the Pl\"ucker coordinates of) $\wedge^2 \C^n$, cf.~\eqref{eq:wedge4sym2wedge2}. 
These quadratic forms generate the \emph{homogeneous vanishing ideal} $I_{\Gr(n)}\subset\C[x_{ij}]$ of $\Gr(n)$, and the \emph{homogeneous coordinate ring} $\C[\Gr(n)]$ of $\Gr(n)$  is given by $\C[x_{ij}]/I_{\Gr(n)}$. The rings $\C[x_{ij}]$ and $\C[\Gr(n)]$, as well as the ideal $I_{\Gr(n)}$, have natural graded structures; as usual, we denote their degree $d$ part with the  subscript $_d$.

The Grassmannian of $2$-planes in $\R^n$ is the set $\Gr(n)(\R)$ of real points of $\Gr(n)$, which we also denote $\Gr(\R^n)$.
The (oriented) Grassmannian $\Gr^+(\R^n)\subset S^{\binom{n}{2}-1}$ in the Introduction is the double-cover of $\Gr(\R^n)\subset\R P^{\binom{n}{2}-1}$ given by the inverse image under the natural projection map.

We identify symmetric endomorphisms $R\in\Sym^2(\wedge^2\R^n)$ with their associated quadratic form $q_R\in\R[x_{ij}]_2$, which is a polynomial in the Pl\"ucker coordinates $x_{ij}$:
\begin{equation*}
q_R(x_{ij})=\left\langle R\left(\sum_{i<j} x_{ij} e_i\wedge e_j\right),\sum_{i<j} x_{ij} e_i\wedge e_j\right\rangle.
\end{equation*}
To simplify notation, we use the same symbol $R$ for both of these objects, and henceforth identify $\Sym^2(\wedge^2\R^n)=\R[x_{ij}]_2$. Under this identification, the subspace $\wedge^4\R^n\subset\Sym^2(\wedge^2\R^n)$ corresponds to the degree $2$ part $(I_{\Gr(n)})_2\subset\R[x_{ij}]_2$ of the graded ideal $I_{\Gr(n)}$.
In particular, its orthogonal complement $\Sym^2_b(\wedge^2\R^n)$ shall be identified with the quotient $\R[\Gr(n)]_2=\R[x_{ij}]_2/(I_{\Gr(n)})_2$, as claimed in \eqref{eq:ident}.

\begin{equation*}
 \xymatrix{ \wedge^4\R^n\; \ar@{^{(}->}^{\!\!\!\!\!\!\!\!\!\!\!\eqref{eq:wedge4sym2wedge2}}[r] \ar@{=}[d] & \Sym^2(\wedge^2\R^n) \ar@{=}[d] \ar@{->>}^{}[r]  & \Sym^2_b(\wedge^2\R^n)\ar@{=}[d] \\
(I_{\Gr(n)})_2 \,\ar@{^{(}->}[r] & \R[x_{ij}]_2 \ar@{->>}[r] &\R[\Gr(n)]_2 }
\end{equation*}

The \emph{sectional curvature} function $\sec_R\colon \Gr^+(\R^n)\to\R$ determined by a (modified) curvature operator $R\in\R[x_{ij}]_2$ is its restriction to $\Gr^+(\R^n)\subset\wedge^2\R^n$, cf.~\eqref{eq:sec}. Since $\sec_R$ is invariant under the antipodal map on $S^{\binom{n}{2}-1}$, it descends to a function on $\Gr(\R^n)$ also denoted by $\sec_R$. 

\begin{definition}\label{def:rseckn}
Given $k\in\R$ and $n\geq2$, let 
\begin{equation*}
\mathfrak R_{\sec\geq k}(n)=\big\{R\in\Sym^2_b(\wedge^2\R^n):\sec_R\geq k\big\},
\end{equation*}
and similarly for $\mathfrak R_{\sec> k}(n)$, $\mathfrak R_{\sec\leq k}(n)$, and $\mathfrak R_{\sec< k}(n)$.
\end{definition}

\begin{remark}\label{rem:obvious}
In order to simplify the exposition, we henceforth consider only the sectional curvature bound $\sec\geq0$, as other sectional curvature bounds can be easily recovered using the following elementary properties:
\begin{enumerate}[(i)]
\item $\mathfrak R_{\sec\geq k}(n)=\overline{\mathfrak R_{\sec> k}(n)}$ and $\mathfrak R_{\sec\leq k}(n)=\overline{\mathfrak R_{\sec< k}(n)}$;
\item $\operatorname{int}\!\big(\mathfrak R_{\sec\geq k}(n)\big)=\mathfrak R_{\sec> k}(n)$ and $\operatorname{int}\!\big(\mathfrak R_{\sec\leq k}(n)\big)=\mathfrak R_{\sec< k}(n)$;
\item $\sec_R\geq k$ if and only if $\sec_{R-k\id}\geq0$, and $\sec_R\leq k$ if and only if $\sec_{k\id-R}\geq0$. In particular, $\mathfrak R_{\sec\geq k}(n)$ and $\mathfrak R_{\sec\leq k}(n)$ are affine images of $\mathfrak R_{\sec\geq 0}(n)$.
\end{enumerate}
The above are direct consequences of linearity of $R\mapsto\sec_R$ and $\sec_{\id} \equiv1$.
\end{remark}

\subsection{Discriminants}\label{subsec:disc}
Given a polynomial $p(x)=a_nx^n+\dots+a_1x+a_0\in\C[x]$, the \emph{discriminant} of $p(x)$ is
a polynomial $\disc_x(p(x))\in\Z[a_0,\dots,a_n]$ with integer coefficients whose variables are the coefficients of $p(x)$, defined as
\begin{equation}\label{eq:disc}
\disc_x(p(x))=a_n^{2n-2}\prod_{i<j}(r_i-r_j)^2,
\end{equation}
where $r_1,\dots,r_n\in\C$ are the roots of $p(x)$. 
It can be computed explicitly in terms of $a_i$, $0\leq i\leq n$, by taking the determinant of the \emph{Sylvester matrix} of $p(x)$ and $p'(x)$.
Clearly, $\disc_x(p(x))=0$ if and only if $p(x)$ has a root of multiplicity $\geq2$.
Note that $\disc_x(-p(x))=\disc_x(p(x))=\disc_x(p(-x))$. 
The discriminant of an $n\times n$-matrix $A$ is defined as the discriminant of its characteristic polynomial, 
 that is, $\disc(A)=\disc_x(\det(A-x\id))$. Thus, $\disc(A)=0$ if and only if $A$ has an eigenvalue of algebraic multiplicity~$\geq2$. 
For a description of $\disc(A)$ as a determinant, see \cite{parlett}.
Irreducibility of $\disc(A)$ for (symmetric) matrices is studied in Appendix~\ref{appendix}.

\subsection{Spectrahedra and their shadows}
We now recall basic notions from convex algebraic geometry, mostly without proofs. As a reference, we recommend~\cite{SIAMbook}.

\begin{definition}
A \emph{spectrahedron} is a set $S\subset \R^n$ of the form
\begin{equation}\label{eq:spectrahedron}
S=\left\{x\in\R^n : \; A+\sum_{i=1}^n x_i B_i\succeq 0\right\},
\end{equation}
where $A, B_i\in\Sym^2(\R^d)$ are symmetric matrices, and $M\succeq0$ means $M$ is positive-semidefinite.
A \emph{spectrahedral shadow} is the image of a spectrahedron under a linear projection, i.e., a set $S\subset\R^n$ of the form
\begin{equation*}
S=\left\{x\in\R^n : \exists\, y\in\R^m, \; A+\sum_{i=1}^n x_i B_i+\sum_{j=1}^m y_j C_j \succeq 0\right\}.
\end{equation*}
\end{definition}

\begin{remarks}\label{rem:specrem}
The following are basic facts about spectrahedra and their shadows:
\begin{enumerate}[(i)]
\item Both spectrahedra and their shadows are convex semialgebraic sets. Furthermore, spectrahedra are closed;
\item The class of spectrahedral shadows contains linear subspaces, polyhedra and all closed convex semialgebraic sets of dimension two~\cite{scheider1};
\item The class of spectrahedral shadows is closed under intersections, linear projections, convex duality, (relative) interior, and closure.
\end{enumerate}
\end{remarks}

The following notion was introduced by Helton and Vinnikov~\cite{heltonvinnikov}, and used in their proof of a conjecture of Peter Lax from 1958, see also \cite[Sec.~6.2.2]{SIAMbook}.

\begin{definition}\label{def:alginterior}
 A closed subset $S\subset \R^n$ is called an \emph{algebraic interior} if it is the closure of a connected component of the set $\{x\in\R^n:\,p(x)>0\}$ for some polynomial $p\in\R[x_1,\ldots,x_n]$, which is called a \emph{defining polynomial} of $S$.
\end{definition}

\begin{remark}\label{rem:definingpoly}
Let $S$ be an algebraic interior, and $p$ be a defining polynomial of minimal degree. Then $p$ divides every defining polynomial of $S$, see~\cite[Lemma 2.1]{heltonvinnikov}. In particular, defining polynomials of minimal degree are unique up to a positive constant factor.
\end{remark}

\begin{lemma}\label{lemma:spectralgint}
Every spectrahedron $S$ with nonempty interior $\operatorname{int}(S)$ is an algebraic interior whose minimal defining polynomial $p$ satisfies $p(x)\neq0$ for all $x\in\operatorname{int}(S)$.
\end{lemma}

\begin{proof}
Assume (without loss of generality) that $d$ in the semidefinite representation \eqref{eq:spectrahedron} of $S$ is minimal.
We claim that $A+\sum_{i=1}^n x_i B_i\succ0$ for all $x\in\operatorname{int}(S)$. Indeed, suppose that this does not hold at some $x_*\in\operatorname{int}(S)$. We may assume that $x_*=0$ and $e_1\in\ker A$. 
This implies that $A$ and $B_i$ are of the form
\begin{equation*}
A=\begin{pmatrix}
0 & 0\\
0 & A'
\end{pmatrix}, \qquad
B_i=\begin{pmatrix}
b_i & v_i^{\mathrm t}\\
v_i & B_i'
\end{pmatrix},
\end{equation*}
where $A', B_i'\in\Sym^2(\R^{d-1})$, $b_i\in\R$, and $v_i\in\R^{d-1}$ is a column vector. Since $0\in\operatorname{int}(S)$, it follows that
$\sum_{i=1}^n x_i b_i\geq0$ for all $x\in\R^n$ near $0$, so
$b_i=0$, $1\leq i\leq n$. Then, applying the same reasoning to the appropriate $2\times 2$-submatrices, it follows that $v_i=0$, $1\leq i \leq n$. Therefore, $S$ admits the semidefinite representation $S=\left\{x\in\R^n : \; A'+\sum_{i=1}^n x_i B'_i\succeq 0\right\}$, contradicting the minimality of $d$.

The spectrahedron $S$ is hence an algebraic interior 
with defining polynomial $\det\!\big(A+\sum_{i=1}^n x_i B_i\big)$, which is positive in $\operatorname{int}(S)$. In particular, the minimal defining polynomial $p(x)$ of $S$ is also positive in $\operatorname{int}(S)$, see Remark~\ref{rem:definingpoly}.
\end{proof}

\begin{example}
It is easy to show that the convex hull of two disjoint unit discs in the plane is not an algebraic interior (this set is called the \emph{football stadium}). It is clearly a spectrahedral shadow, and not a spectrahedron by Lemma~\ref{lemma:spectralgint}.
\end{example}

\begin{example}
The \emph{TV screen} $\{(x,y)\in\R^2:\, x^4+y^4\leq 1\}$ is an algebraic interior and a spectrahedral shadow, but not a spectrahedron.
\end{example}

\subsection{Quadratic forms on real projective varieties}\label{ssec:quadforms}
Let $X\subset\C P^N$ be a real projective variety that is irreducible, not contained in any hyperplane, and whose set of real points $X(\R)\subset\R P^N$ is Zariski dense in $X$. Let $\R[X]$ be its homogeneous coordinate ring, i.e., the polynomial ring in $N+1$ variables modulo the homogeneous vanishing ideal of $X$. Inside the degree $2$ part $\R[X]_2$, we consider the convex cone
\begin{equation}\label{eq:SigmaX}
\Sigma_X=\big\{f\in\R[X]_2:\,f=g_1^2+\ldots+g_s^2, \; g_i\in\R[X]_1\big\}
\end{equation}
consisting of all sums of squares of linear forms, as well as the convex cone
\begin{equation}\label{eq:PX}
\textrm{P}_X=\big\{f\in\R[X]_2:\,f\geq0 \text{ on } X(\R) \big\}
\end{equation}
of all \emph{nonnegative} quadratic forms on the real points $X(\R)$. For the latter, note that the sign of $f$ at any point $x\in X(\R)$ is well-defined because $f$ has even degree. Clearly, $\Sigma_X\subset\textrm{P}_X$. Blekherman, Smith, and Velasco~\cite[Thm.~1]{blekhermanetal} showed that $\Sigma_X=\textrm{P}_X$ if and only if $X$ has minimal degree, namely $\deg(X)=\textrm{codim}(X)+1$. It is worth pointing out that $\textrm{P}_X$ and $\Sigma_X$ are pointed closed convex cones with nonempty interior. 
Moreover, $\Sigma_X$ is a spectrahedral shadow: it is the image of the cone of positive-semidefinite matrices under the linear map sending a symmetric matrix $A$ to  $x^{\mathrm t} A x\in\R[X]_2$ where $x=(x_0,\ldots,x_N)^{\mathrm t}$.

In the special case in which $X\subset\C P^N$ is a \emph{quadric}, that is, the zero set of a single quadratic form, the above conclusion $\Sigma_X=\textrm{P}_X$ holds without any additional assumptions and has been known for almost 100 years~\cite{finsler}.

\begin{lemma}[Finsler]\label{lemma:finsler}
Let $A,B\in\Sym^2(\R^d)$. The following are equivalent:
\begin{enumerate}[\rm (i)]
\item $\langle Av,v\rangle\geq0$ for all $v\in\R^d\setminus\{0\}$ such that $\langle Bv,v\rangle=0$;
\item there exists $x\in\R$ such that $A+xB\succeq0$.
\end{enumerate}
The analogous statement replacing all inequalities by strict inequalities also holds. Moreover, if {\rm (i)} holds and there exists $v\in\R^d\setminus\{0\}$ with $\langle Av,v\rangle=\langle Bv,v\rangle=0$ and $B$ has full rank, then $x$ in {\rm (ii)} is unique.
\end{lemma}

Calabi \cite{calabi}  independently found an elegant topological proof of this result; see also the survey~\cite{kursad} for other proofs and discussion.

\begin{example}\label{ex:grasspsdsos}
The sets $\textrm{P}_X$ and $\Sigma_X$ have 
an important geometric interpretation
when $X=\Gr(n)$ is the Grassmannian of $2$-planes  \eqref{eq:grassm}.
Namely, keeping in mind the identifications described in Subsection~\ref{subsec:grasscurvop}, we have that
\begin{equation}\label{eq:pgrr}
\textrm{P}_{\Gr(n)}=\mathfrak R_{\sec\geq0}(n),
\end{equation}
and $\Sigma_{\Gr(n)}$ is the set of curvature operators with \emph{strongly nonnegative curvature}, see~\cite{strongnonneg,strongpos}. 

In order to compare $\textrm{P}_{\Gr(n)}$ and $\Sigma_{\Gr(n)}$ using the results in \cite{blekhermanetal}, note that
 \begin{equation*}
\operatorname{codim} \Gr(n)=\binom{n}{2}-1-2 (n-2)=\frac{(n-2)(n-3)}{2},
\end{equation*}
and, by \cite[Ex.~4.38]{3264},
\begin{equation*}
\deg\Gr(n)=\frac{(2(n-2))!}{(n-2)! (n-1)!}.
\end{equation*}
Thus, $\deg\Gr(n)=\operatorname{codim}\Gr(n)+1$
if and only if $n\leq4$.
Therefore, by Blekherman, Smith, and Velasco~\cite[Thm.~1]{blekhermanetal},
$\Sigma_{\Gr(n)}=\textrm{P}_{\Gr(n)}$ if and only if $n\leq4$. In particular, $\textrm{P}_{\Gr(n)}$ is a spectrahedral shadow for $n\leq4$. On the other hand, for all $n\geq5$, there exists a nonnegative quadratic form $P\in\textrm{P}_{\Gr(n)}$ that is not a sum of squares, which plays a crucial role in the proof of Theorem~\ref{mainthm:spectrahedra}~\eqref{A3}. An explicit example of such $P\in\textrm{P}_{\Gr(n)}\setminus\Sigma_{\Gr(n)}$, $n\geq5$, found by Zoltek~\cite{Zoltek79}, is given by:
\begin{equation*}
R_{\rm Zol}=x_{1 2}^2+2  x_{1 3}^2+2  x_{2 3}^2+2  x_{1 4}^2 +x_{1 5}^2+x_{3 4}^2 +2  x_{2 5}^2+2  x_{4 5}^2 -2  x_{1 2}  x_{3 4}-2  x_{1 2}  x_{1 5}-2  x_{3 4}  x_{1 5}.
\end{equation*}

\end{example}

\begin{remark}
Since real symmetric matrices are diagonalizable, $R\in \R[x_{ij}]_2$ is a sum of squares
if and only if $R\succeq0$, i.e.,
\begin{equation*}
\Sigma_{\C P^{\binom{n}{2} -1}}=\textrm{P}_{\C P^{\binom{n}{2} -1}}
\end{equation*}
is the cone of positive-semidefinite curvature operators.
\end{remark}

\subsection{Scheiderer's criterion}
The only methods currently available to prove that a certain convex semialgebraic set is \emph{not} a spectrahedral shadow were recently developed in a seminal work by Scheiderer~\cite{claus}. For example, he showed that $\textrm{P}_X$ is not a spectrahedral shadow when $X\subset\C P^{\binom{n+d}{d}-1}$ is the $d$th Veronese embedding of $\C P^n$ for every $n\geq2,\, d\geq3$; or $n\geq3,\, d\geq2$. Note that these are exactly the cases in which the $d$th Veronese embedding does not satisfy $\deg(X)=\textrm{codim}(X)+1$, see~\cite[Sec.~2.1.2]{3264}

In order to present Scheiderer's criterion, recall the following basic definitions and facts of convex geometry.
Given any subset $\Omega$ of a finite-dimensional real vector space $V$, the \emph{convex hull} and \emph{conic hull} of $\Omega$ are defined, respectively, as: 
\begin{equation*}
\begin{aligned}
\conv(\Omega)&=\left\{\textstyle\sum\limits_{i=1}^k \alpha_i v_i : v_i\in\Omega, \alpha_i\geq0, \textstyle\sum\limits_{i=1}^k \alpha_i=1, \, k\in\mathds N \right\}, \\
\cone(\Omega)&=\left\{\textstyle\sum\limits_{i=1}^k \alpha_i v_i : v_i\in\Omega, \alpha_i\geq0, \, k\in\mathds N \right\}.
\end{aligned}
\end{equation*}
Denote by $V^\vee$ the dual vector space of $V$.
The \emph{cone dual} of $\Omega$ is defined as
\begin{equation}\label{eq:conedual}
\Omega^*=\{\lambda\in V^\vee : \lambda(x)\geq0, \, \forall x\in \Omega\}
\end{equation}
and satisfies the following properties:
\begin{enumerate}[\rm (i)]
\item $\Omega^*\subset V^\vee$ is a closed convex cone;
\item $\Omega^*=\Big(\overline{\conv(\Omega)}\Big)^*$ and similarly for cone duals of any combinations of conic hull, convex hull, and closure;
\item\label{eq:bidual} $\Omega^{**}=(\Omega^*)^*=\overline{\cone(\Omega)}$;
\item $\Omega^*$ is a spectrahedral shadow whenever $\Omega$ is a spectrahedral shadow.
\end{enumerate}
For properties (iii) and (iv), see e.g.~\cite[(5.11)]{SIAMbook} and \cite[Thm.~5.57]{SIAMbook}, respectively.

\begin{definition}
Given $f\in\R[x_1,\dots,x_n]$, its \emph{homogenization} is
the unique homogeneous polynomial $f^h\in\R[t,x_1,\dots,x_n]$ with the same degree as $f$  such that $f^h(1,x_1,\dots,x_n)=f(x_1,\dots,x_n)$.
\end{definition}

We are now in position to present the following convenient criterion to check if a semialgebraic set is not a spectrahedral shadow, extracted from Scheiderer~\cite{claus}.

\begin{theorem}\label{thm:claus}
 Let $L\subset\R[x_1,\ldots,x_n]$ be a finite-dimensional vector space with $1\in L$, and let $f\in\R[x_1,\dots,x_n]$ be a nonnegative polynomial which is not a sum of squares. Suppose that, for all $y\in\R^n$, the coefficients of $f^h(t,x_1-y_1,\ldots,x_n-y_n)$, considered as a polynomial in the new variable $t$, belong to $L$. Then the set 
\begin{equation*}
K=\{g\in L:\, g(x)\geq0 \, \text{ for all } x\in\R^n\} 
\end{equation*}
is not a spectrahedral shadow.
\end{theorem}

\begin{proof}
For each $x\in\R^n$, let $\phi_x\in L^\vee$ be the evaluation functional $\phi_x(f)=f(x)$. Note that $K=\{\phi_x,\,x\in\R^n\}^*=C^*$, where
\begin{equation*}
C=\overline{\conv(\{\phi_x, \, x\in\R^n\})}\subset L^\vee.
\end{equation*}
Since $f$ is nonnegative but not a sum of squares, the same holds for its homogenization $f^h$. Using this and that the coefficients of $f^h(t,x_1-y_1,\ldots,x_n-y_n)$ belong to $L$, one can show, by the exact same reasoning 
as in \cite[Ex.~4.20, Rem.~4.21]{claus}, which follow from \cite[Lem.~4.17, Prop.~4.18, Prop.~4.19]{claus}, that
$C$ is not a spectrahedral shadow.

Suppose, by contradiction, that $K$ is a spectrahedral shadow, so that its cone dual $K^*$ is also a spectrahedral shadow. By \eqref{eq:bidual}, we have that $K^*=C^{**}=\overline{\cone(C)}$. Since $1\in L$, every evaluation $\phi_x$ is contained the affine hyperplane $H=\{\lambda\in L^\vee : \lambda(1)=1\}$ and hence $C\subset H$. In particular, $C=\overline{\cone(C)}\cap H=K^*\cap H$ is a spectrahedral shadow, providing the desired contradiction.
\end{proof}

\section{Curvature operators in dimension 4}
\label{sec:dim4}

In this section, we prove Theorem~\ref{mainthm:4dim} and statement \eqref{A2} in Theorem~\ref{mainthm:spectrahedra}.
For simplicity, we only treat the case $\sec\geq0$ for the reasons discussed in Remark~\ref{rem:obvious}.
Furthermore, we denote by $\mathfrak p(R)$ the polynomial $\mathfrak p_0(R)$ from \eqref{eq:pk}.

\subsection{Curvature operators on 4-manifolds}
Recall that the \emph{Hodge star} operator is defined as $*\in\Sym^2(\wedge^2\R^4)$ corresponding to $\omega=e_1\wedge e_2\wedge e_3\wedge e_4\in\wedge^4\R^4$ under the identification \eqref{eq:wedge4sym2wedge2}. Its eigenvalues are $\pm 1$ and the corresponding eigenspaces $\wedge^2_\pm\R^4\cong\R^3$ consist of so-called \emph{self-dual} and \emph{anti-self-dual} $2$-forms. Any symmetric endomorphism $R\colon\wedge^2 \R^4\to\wedge^2 \R^4$ can be represented by a block matrix with respect to the decomposition $\wedge^2 \R^4=\wedge^2_+\R^4\oplus\wedge^2_-\R^4\cong\R^6$, 
\begin{equation}\label{eq:Rblocks}
R=\begin{pmatrix}
A & B \\
B^{\rm t} & C
\end{pmatrix}
\end{equation}
where $A$ and $C$ are symmetric $3\times 3$-matrices and $B$ is any $3\times 3$-matrix.
Note that
\begin{equation}\label{eq:hodgestar}
*=\begin{pmatrix}
\id & 0 \\
0 & -\id
\end{pmatrix},
\end{equation}
and the Bianchi identity \eqref{eq:bianchi} for $R$ as in \eqref{eq:Rblocks} is
$
\tr A-\tr C=\tr(R \, *)=\langle R,*\rangle=0.
$
Finally, denote by $\Sym^2(\C^6)$ the space of complex symmetric $6\times 6$-matrices. 

The following characterization of $\mathfrak R_{\sec\geq 0}(4)$ was given by Thorpe~\cite{Thorpe72}. Note that this an immediate application of Finsler's Lemma~\ref{lemma:finsler}. 

\begin{proposition}[Thorpe's trick]\label{prop:thorpe}
A curvature operator $R\in\Sym^2_b(\wedge^2\R^4)$ has $\sec_R\geq0$ if and only if there exists $x\in\R$ such that $R+x\,*\succeq0$, and analogously for $\sec_R>0$. Moreover, if $R$ has $\sec_R\geq0$ but not $\sec_R>0$, then $x\in\R$ is unique.
\end{proposition}

\subsection{\texorpdfstring{Vanishing of $\mathfrak p$ on $\partial \mathfrak R_{\sec\geq 0}(4)$}{Vanishing on the boundary}}
We make repeated use of the complex matrix
\begin{equation*}
T=\diag\big(1,1,1,\sqrt{-1},\sqrt{-1},\sqrt{-1}\,\big),
\end{equation*}
written in terms of the above identification \eqref{eq:Rblocks}.
Note that $T$ is a square root of $*$ and hence $\det(R+x*)=-\det(TRT+x\id)$, since $T(R+x*)T=TRT+x\id$ and $\det(T^2)=-1$. 
Thus, $\disc_x\!\big(\!\det(R+x*)\big)=\disc(TRT)$ by the properties of discriminants:
\begin{equation}\label{eq:Ttrick}
\begin{aligned}
\disc_x\!\big(\!\det(R+x*)\big)&=\disc_x\!\big(\!-\det(TRT+x\id)\big)\\
&=\disc_x(\det(TRT-x\id))\\
&=\disc(TRT).
\end{aligned}
\end{equation}

\begin{proposition}\label{prop:zeroboundary}
The polynomial $\mathfrak p(R)=\disc_x(\det(R+x\,*))$ vanishes on the (topological) boundary $\partial\, \mathfrak R_{\sec\geq0}$.
\end{proposition}

\begin{proof}
Let $R\in\partial\, \mathfrak R_{\sec\geq0}$, which means that $\sec_R\geq0$ but $R$ does not satisfy $\sec_R>0$.
By the Finsler--Thorpe trick (Proposition \ref{prop:thorpe}), there exists $x_0\in\R$ such that $R+x_0\,*\succeq0$. It suffices to show  $x_0$ is a root of $\det(R+x\,*)$ with multiplicity $\geq2$.

Assume that $x_0$ is a simple root of this polynomial.
Then $-x_0$ is an eigenvalue of $TRT$ with algebraic (and hence geometric) multiplicity $1$. Since 
$R+x_0\,*=T^{-1}(TRT+x_0\id)T^{-1}$,
this implies that $\dim\ker(R+x_0\,*)=1$. As $R+x_0\,*\succeq0$, it has $5$ positive eigenvalues (counted with multiplicities). 
Since $x_0$ was assumed to be a simple root of $\det(R+x\,*)$, there exists $x$ near $x_0$ such that $\det(R+x\,*)>0$ and hence $R+x\,*\succ0$, contradicting the fact that $R$ does not satisfy $\sec_R>0$.
\end{proof}

\begin{remark}
The discriminant $\disc_x\!\big(\!\det(R+x\,*)\big)$ is a homogeneous polynomial of degree 30 in the coefficients of $R$.
\end{remark}

\begin{remark}
As a consequence of the above proof, $-x_0$ is an eigenvalue of $TRT$ with \emph{algebraic} multiplicity $\geq2$. We warn the reader that this does not imply that its \emph{geometric} multiplicity is $\geq2$, because $TRT$ is a \emph{complex} symmetric matrix, hence not necessarily diagonalizable.

In fact, the geometric multiplicity of $-x_0$ \emph{cannot} be always $\geq2$. Indeed, on the one hand, $\operatorname{codim}\!\big(\partial\, \mathfrak R_{\sec\geq0}\big)=1$. On the other hand, as in the proof of Proposition~\ref{prop:codiminterior} below, since the set of symmetric $6\times 6$-matrices with rank $\leq4$ has codimension $\geq3$ (see \cite[p.~72]{harris-tu}), it follows that
\begin{equation*}
\operatorname{codim}\!\big(\{R\in\Sym^2_b(\wedge^2 \R^4):\exists\, x_0\in\R, \,\operatorname{rank}(R+x_0\,*)\leq 4\}\big)\geq2.
\end{equation*}
\end{remark}

\subsection{\texorpdfstring{Zeroes of $\mathfrak p$ in $\mathfrak R_{\sec>0}(4)$}{Zeroes in the interior}}
We now study the interior vanishing locus of $\mathfrak p$.

\begin{proposition}\label{prop:codiminterior}
The zero set $\{R\in\mathfrak R_{\sec>0}(4):\mathfrak p(R)=0\}$ is a real subvariety of codimension $\geq2$, hence its complement $\{R\in\mathfrak R_{\sec>0}(4):\mathfrak p(R)>0\}$ is connected.
\end{proposition}

\begin{proof}
Consider the orthogonal projection $\pi \colon\Sym^2(\wedge^2\R^4)\to\Sym^2_b(\wedge^2\R^4)$. 
Since the set of matrices of rank $\leq4$ in $\Sym^2(\C^6)$ is a subvariety of codimension $3$, see e.g., \cite[p.~72]{harris-tu}, the real subvariety $\mathcal Y=\{R\in \Sym^2(\wedge^2\R^4):\operatorname{rank}(R)\leq 4\}$ has codimension $\geq3$.
Thus, $\pi(\mathcal Y)$  has codimension $\geq2$ since $\dim \pi(\mathcal Y)\leq\dim \mathcal Y$ and the ambient dimension drops by one. Thus, it suffices to show that for all $R\in\mathfrak R_{\sec>0}(4)$ with $\mathfrak p(R)=0$, there is $x_0\in\R$ with $\operatorname{rank}(R+x_0\,*)\leq4$, i.e., $R\in\pi(\mathcal Y)$.

Choose $\lambda\in\R$ such that $R+\lambda\, *\succ0$ and an invertible real $6\times 6$-matrix $S$ such that $R+\lambda\, *=S^\mathrm tS$.
Since $\mathfrak p(R)=0$, the polynomial $q(t)=\det(R+(\lambda+t)\,*)$ has a root $t_0\neq0$ of multiplicity $\geq2$. 
We claim that $\operatorname{rank}(R+x_0\,*)\leq4$ for $x_0=\lambda+t_0$. Note that
\begin{equation}\label{eq:trick}
\begin{aligned}
q(t)&= \det(S^\mathrm tS+t\,*)\\
&=(\det S)^2\det\!\big(\id+t \, (S^\mathrm t)^{-1} * S^{-1}\big)\\
&=(\det S)^2 t^6\det\!\left(\tfrac1t\id+ (S^\mathrm t)^{-1} * S^{-1}\right)\\
&=(\det S)^2 \tfrac{1}{s^6}\det\!\left((S^\mathrm t)^{-1} * S^{-1}-s\id\right),
\end{aligned}
\end{equation}
where $s=-\frac1t$. Since $q(t_0)=0$ and $q'(t_0)=0$, it follows that $-\frac{1}{t_0}$ is an eigenvalue of $(S^\mathrm t)^{-1} * S^{-1}$ with algebraic multiplicity $\geq2$. As this is a real symmetric matrix, the geometric multiplicity of $-\frac{1}{t_0}$ is also $\geq2$. 
Thus,
\begin{align*}
\operatorname{rank}\!\left(R + (\lambda+t_0)*\right)&=
\operatorname{rank}\!\left(\tfrac{1}{t_0}(R+\lambda\,*)+*\right)\\
&=\operatorname{rank}\!\left(* +\tfrac{1}{t_0} S^\mathrm t S\right)\\
&=\operatorname{rank}\!\left((S^\mathrm t)^{-1} * S^{-1}+\tfrac{1}{t_0} \id\right)\leq 4.\qedhere
\end{align*}
\end{proof}

\begin{remark}
It follows from \eqref{eq:trick} that the polynomial $t\mapsto\det(R+t\,*)$ only has real roots if $R\in\mathfrak R_{\sec\geq0}(4)$, since $(S^\mathrm t)^{-1} * S^{-1}$ is a real symmetric matrix. In particular, its discriminant $\mathfrak p(R)$ is nonnegative, see \eqref{eq:disc}.
\end{remark}

\subsection{\texorpdfstring{$\mathfrak R_{\sec\geq0}(4)$ as an algebraic interior}{Algebraic interior}}
We now prove Theorems~\ref{mainthm:4dim} and \ref{mainthm:spectrahedra} \eqref{A2}.

\begin{proof}[Proof of Theorem~\ref{mainthm:4dim}]
The zero set $\{\mathfrak p(R)=0\}$ contains the topological boundary of $\mathfrak R_{\sec\geq 0}(4)$ by Proposition~\ref{prop:zeroboundary} and has codimension $\geq2$ in its interior by Proposition~\ref{prop:codiminterior}.
By direct inspection, $\mathfrak p(R)>0$ at $R=\operatorname{diag}(1,2,3,4,5,6)\in\mathfrak R_{\sec\geq0}(4)$.
This implies that $\mathfrak R_{\sec\geq 0}(4)$ is an algebraic interior with defining polynomial $\mathfrak p(R)$, see Definition~\ref{def:alginterior}.
We claim that this polynomial $\mathfrak p\colon\Sym^2_b(\wedge^2 \R^4)\to\R$ is irreducible over $\R$, and hence it is a minimal defining polynomial, see Remark~\ref{rem:definingpoly}.

Denote by $\widetilde{\mathfrak p}\colon \Sym^2(\wedge^2\R^4)\to\R$ the polynomial $\widetilde{\mathfrak p}=\mathfrak p\circ\pi$, where $\pi$ is the orthogonal projection onto $\Sym^2_b(\wedge^2\R^4)$. Clearly, $\mathfrak p$ is irreducible if and only if $\widetilde{\mathfrak p}$ is irreducible. On the other hand, $\widetilde{\mathfrak p}(R)$ is given by the same formula \eqref{eq:pk} as $\mathfrak p(R)$ since shifts in the variable $x$ do not change the discriminant, see \eqref{eq:disc}.
It suffices to show that the complexification $\widetilde{\mathfrak p}\colon\Sym^2(\C^6)\to\C$ of $\widetilde{\mathfrak p}$ is irreducible over $\C$. 
This is a consequence of Proposition~\ref{prop:discirred} in the Appendix, because $\widetilde{\mathfrak p}(R)$ is the discriminant of $TRT$ according to \eqref{eq:Ttrick}, and $R\mapsto TRT$ is a linear isomorphism of $\Sym^2(\C^6)$.

Finally, the last statement in Theorem~\ref{mainthm:4dim} follows from the fact that the closure of any two connected components of $\{\mathfrak p(R)>0\}$ can only intersect at boundary points and $\id\in\operatorname{int}(\mathfrak R_{\sec\geq0}(4))$.
\end{proof}

\begin{proof}[Proof of Theorem~\ref{mainthm:spectrahedra} (2)]
The set $\mathfrak R_{\sec\geq0}(4)$ is a spectrahedral shadow as a consequence of the Finsler--Thorpe trick (Proposition~\ref{prop:thorpe}), see also Example~\ref{ex:grasspsdsos}. Furthermore, it is not a spectrahedron by Lemma~\ref{lemma:spectralgint}, since $\mathfrak p$ is a minimal defining polynomial for $\mathfrak R_{\sec\geq0}(4)$ by Theorem~\ref{mainthm:4dim}, and $\mathfrak p$ vanishes at the interior point $\id\in\mathfrak R_{\sec>0}(4)$, since $\mathfrak p(\id)=\disc_x\!\big(\!\det(\id+x\, *)\big)=\disc_x\!\big((1+x)^3(1-x)^3\big)=0$.
\end{proof}

\section{\texorpdfstring{Curvature operators in dimensions $\geq5$}{Curvature operators in dimensions greater than 5}}\label{sec:notshadow}

Using the notation from Section~\ref{sec:cagprel}, consider the Pl\"ucker  embedding of the Grassmannian $\Gr(5)$ in $\C P^9$, and 
recall that $\textnormal{P}_{\Gr(n)}\subset\R[\Gr(n)]_2$ is the subset of nonnegative quadratic forms on the real points of $\Gr(n)$, see \eqref{eq:PX}. The main step in the proof of Theorem~\ref{mainthm:spectrahedra}~\eqref{A3} is the following application of  Theorem~\ref{thm:claus}.

\begin{proposition}\label{prop:notspec}
The closed convex cone $\textnormal{P}_{\Gr(5)}$ is not a spectrahedral shadow.
\end{proposition}

\begin{proof}
By \cite[Thm.~1]{blekhermanetal} or \cite{Zoltek79}, see Example~\ref{ex:grasspsdsos}, there exists a polynomial $P\in \textnormal{P}_{\Gr(5)}\setminus\Sigma_{\Gr(5)}$. In other words, $P$ is a nonnegative quadratic form that is not a sum of squares modulo the vanishing ideal of $\Gr(5)$, i.e., the ideal generated by the Pl\"ucker relations.
We consider the affine chart $U$ of $\Gr(5)$ defined by the Pl\"ucker coordinate $x_{15}$ being nonzero. Every point in $U$ is a $2$-plane of the form
\begin{equation*}
(e_1+x_{25}e_2+x_{35}e_3+x_{45}e_4)\wedge
(x_{12}e_2+x_{13}e_3+x_{14}e_4+e_5),
\end{equation*}
that is, the row span of the matrix
\begin{equation}\label{eq:gamma}
\Gamma=\begin{pmatrix} 1& x_{25}& x_{35} & x_{45}&0\\ 0& x_{12} & x_{13}&x_{14}&1 \end{pmatrix}.
\end{equation}
Note that $U$ is isomorphic to an affine complex space $\mathds{A}^{6}$ with coordinates $x_{1j}$ and $x_{i5}$, $2\leq i,j\leq 4$. 
Consider the linear map given by restriction from $\R[\Gr(5)]_2$ to the regular functions on $U$:
\begin{equation}
\psi\colon\R[\Gr(5)]_2 \longrightarrow\R[U]=\R[x_{12},x_{13},x_{14},x_{25}, x_{35} , x_{45}]
\end{equation}
that sends every Pl\"ucker coordinate $x_{ij}$ to the $2\times 2$-minor corresponding to the columns $i$ and $j$
of the matrix $\Gamma$. More precisely, 
the effect of applying $\psi$ to an element of $\R[\Gr(5)]_2$ represented by $\sum_{i<j,k<l} R_{ijkl} x_{ij}x_{kl}\in\Sym^2(\wedge^2\R^5)$ consists of making the following substitutions:
\begin{align}\label{eq:substitutions2}
&x_{15}\rightsquigarrow 1 & x_{23}\rightsquigarrow x_{25}x_{13}-x_{35}x_{12}\nonumber \\
&x_{1j}\rightsquigarrow x_{1j}, \; 2\leq j\leq4 & x_{24}\rightsquigarrow x_{25}x_{14}-x_{45}x_{12}\\
&x_{i5}\rightsquigarrow x_{i5}, \; 2\leq i\leq4 & x_{34}\rightsquigarrow x_{35}x_{14}-x_{45}x_{13}\nonumber
\end{align}

Note that since the real points of $U$ are dense in the real points of $\Gr(5)$, we have that $\psi$ is injective, and the subset of nonnegative polynomials contained in its image $L=\psi\big(\R[\Gr(5)]_2\big)$ is exactly $\psi(\textnormal{P}_{\Gr(5)})$. 
In particular, $\textnormal{P}_{\Gr(5)}$ is not a spectrahedral shadow if and only if $K=\{g\in L: g \geq 0 \text{ on } U(\R)\}$ is not a spectrahedral shadow.
We will show that the latter holds applying Theorem~\ref{thm:claus}.

First, note that $1\in L$ since it is the image of $(x_{15})^2$. 

\begin{claim}
The polynomial $f=\psi(P)$ is nonnegative but not a sum of squares.
\end{claim}

Nonnegativity of $f$ follows directly from the fact that $P$ is nonnegative.

For the sake of contradiction, suppose that
\begin{equation}\label{eq:fsos}
f=g_1^2+\ldots+g_r^2
\end{equation}
for some $g_k\in\R[U]$. 
Note that $\deg f\leq 4$, and hence $\deg g_k\leq 2$. Thus, we can write $g_k=l_k+p_k$, where $\deg l_k\leq 1$ and  $p_k$ are homogeneous polynomials of degree two. 
The homogeneous part of degree four of $f$ is then
\begin{equation}\label{eq:f4}
f_4=p_1^2+\ldots+p_r^2,
\end{equation}
and by \eqref{eq:substitutions2} it vanishes at the points where the matrix
\begin{equation}\label{eq:gammaprime}
\Gamma'=\begin{pmatrix} x_{25}& x_{35} & x_{45}\\ x_{12} & x_{13}&x_{14} \end{pmatrix}
\end{equation}
has rank at most $1$, since this is equivalent to the vanishing of its $2\times 2$ minors.
By \eqref{eq:f4},
this implies that every $p_k$ also vanishes at these points.

The ideal $I$ of $\C[x_{12},x_{13},x_{14},x_{25}, x_{35} , x_{45}]$ generated by the $2\times 2$ minors of $\Gamma'$ is prime and hence radical, see e.g.~\cite[Thm.~2.10]{minors}. Since the real points are Zariski dense in the complex affine variety defined by $I$, every $p_k$ vanishes on this variety. By Hilbert's Nullstellensatz, $p_k$ is a linear combination of the $2\times2$ minors of $\Gamma'$, with \emph{real} coefficients because $p_k$ are real. Since all of $x_{1j}$, $x_{i5}$, $2\leq i,j\leq 4$ and $1$ are $2\times2$ minors of $\Gamma$, we have that $g_k=l_k+p_k$ is a linear combination of $2\times2$ minors of $\Gamma$.
As the square of a linear combination of 
$2\times2$ minors of $\Gamma$ is the image by $\psi$ of a square in $\R[\Gr(5)]_2$, 
injectivity of $\psi$ and \eqref{eq:fsos} 
contradict the fact that $P$ is not a sum of squares, concluding the proof of the above Claim.

\medskip

The homogenization $f^h$ can be obtained substituting each Pl\"ucker coordinate $x_{ij}$ in $P$ by the $2\times 2$-minor corresponding to columns $i$ and $j$ of the matrix
\begin{equation*}
\begin{pmatrix} t& x_{25}& x_{35} & x_{45}&0\\ 0& x_{12} & x_{13}&x_{14}&t \end{pmatrix}.
\end{equation*}
This and multilinearity of the determinant imply that all coefficients of $$f^h(t,x_{12}-y_{12},\ldots,x_{45}-y_{45}),$$ considered as a polynomial in $t$, belong to $L$ for all $y_{ij}\in\R$. Thus, by Theorem~\ref{thm:claus},
we have that $\psi\big(\textrm{P}_{\Gr(5)}\big)$, and hence $\textrm{P}_{\Gr(5)}$, are not spectrahedral shadows.
\end{proof}

Although our geometric applications only require the following result for 
$\Gr(n)$, $n\geq5$, 
for the sake of completeness, we state and prove it in the more general case of the Grassmannian $\textnormal{Gr}_k(n)$ of $k$-planes in $n$-dimensional space.

\begin{corollary}\label{cor:notspec}
The closed convex cone $\textnormal{P}_{\textnormal{Gr}_k(n)}$ is not a spectrahedral shadow for all $n\geq5$ and $2\leq k\leq n-2$.
\end{corollary}

\begin{proof}
We proceed by induction on $n\geq5$. For $n=5$, the conclusion holds by Proposition~\ref{prop:notspec}, since  $\textnormal{Gr}_2(5)\cong \textnormal{Gr}_{3}(5)$ are naturally isomorphic.

For the induction step, fix $n\geq5$ and suppose $\textnormal{P}_{\textnormal{Gr}_{k}(n)}$ is not a spectrahedral shadow for all $2\leq k\leq n-2$. 
Since $\textnormal{Gr}_k(n+1)\cong \textnormal{Gr}_{n-k}(n+1)$ are naturally isomorphic, it suffices to show that $\textnormal{P}_{\textnormal{Gr}_{k}(n+1)}$ is not a spectrahedral shadow for all $2\leq k\leq \tfrac{n+1}{2}$. Note that every such $k$ satisfies $2\leq k\leq n-2$ because $n\geq5$.

The cone on the Grassmannian of $k$-planes in $\C^n$,
\begin{equation}
C\textnormal{Gr}_k(n)=\big\{v_1\wedge\dots\wedge v_k : v_i\in\C^n\big\}\subset \wedge^k\C^n,
\end{equation}
is an affine variety whose (affine) coordinate ring agrees with the homogeneous coordinate ring of $\textnormal{Gr}_k(n)$.
The projection map $\pi\colon\C^{n+1}\to\C^n$ onto the first $n$ coordinates induces a linear map $\wedge^k\pi\colon \wedge^k\C^{n+1}\to\wedge^k\C^n$, which maps $C\textnormal{Gr}_k(n+1)$ onto $C\textnormal{Gr}_k(n)$. In particular, the associated homomorphism of coordinate rings given by composition with $\wedge^k\pi$ is injective. This implies that $\R[\textnormal{Gr}_k(n)]_2$ can be identified with a linear subspace of $\R[\textnormal{Gr}_k(n+1)]_2$.
Note that $\textnormal{P}_{\textnormal{Gr}_{k}(n)}\subset \R[\textnormal{Gr}_k(n)]_2$ consists of the elements which are nonnegative on the real points $C\textnormal{Gr}_k(n)(\R)$, and similarly for $\textnormal{P}_{\textnormal{Gr}_{k}(n+1)}\subset \R[\textnormal{Gr}_k(n+1)]_2$.
Therefore,
\begin{equation}\label{eq:linearslice}
\textnormal{P}_{\textnormal{Gr}_{k}(n)}=\textnormal{P}_{\textnormal{Gr}_{k}(n+1)}\cap \R[\textnormal{Gr}_k(n)]_2
\end{equation}
because $\wedge^k\pi$ maps $C\textnormal{Gr}_k(n+1)(\R)$ onto $C\textnormal{Gr}_k(n)(\R)$ and pull-back of nonnegative functions are nonnegative.
Since the intersection of a spectrahedral shadow with a linear subspace is also a spectrahedral shadow, it follows from the induction hypothesis and \eqref{eq:linearslice} that $\textnormal{P}_{\textnormal{Gr}_{k}(n+1)}$ is also not a spectrahedral shadow.
\end{proof}

\begin{proof}[Proof of Theorem~\ref{mainthm:spectrahedra} (1)]
We have that $\mathfrak R_{\sec\geq0}(n)=\textrm{P}_{\Gr(n)}$ by \eqref{eq:pgrr}, and this is not a spectrahedral shadow by Corollary~\ref{cor:notspec}.
\end{proof}

\begin{remark}
 Both the cones from Corollary~\ref{cor:notspec} and the counterexamples to the Helton--Nie conjecture by Scheiderer are dual cones to the convex hull of a highest weight orbit of some $\SO(n)$-representation, i.e., they are the dual cones to \emph{orbitopes} \cite{orbitopes}, see also \cite{babl}. For Scheiderer's examples these representations are $\Sym^{2d}(\R^n)$, $d\geq2$, $n\geq3$ and $(d,n)\neq(2,3)$, whereas here it is $\Sym^2(\wedge^k\R^n)$ for $n\geq5$ and $2\leq k\leq n-2$. Note that on the other hand for the representations $\Sym^{2}(\R^n)$, $\Sym^{2d}(\R^2)$ and $\wedge^2\R^n$ these cones are even spectrahedra \cite{orbitopes}.
\end{remark}

\section{Relaxations and Algorithms}
\label{sec:relaxalgo}

We now construct \emph{relaxations} of $\mathfrak R_{\sec\geq0}(n)$, that is, inner and outer approximations, proving Theorem~\ref{mainthm:approx}. These relaxations are then combined with \emph{semidefinite programming} to establish simple algorithms to test whether a given curvature operator $R\in\Sym^2_b(\wedge^2\R^n)$ has $\sec_R\geq0$, or any other sectional curvature bound, after a simple modification (Remark~\ref{rem:obvious}), see Algorithm~\ref{alg:relax}. 
For $n=4$, we exploit the fact that $\mathfrak R_{\sec\geq0}(4)$ is a spectrahedral shadow 
to construct a more efficient algorithm based instead on \emph{Sturm's real root counting}, see Algorithms~\ref{alg:poscurv} and \ref{alg:nonnegcurv}.

\subsection{Inner relaxations}
Ideas similar to the Lasserre hierarchy~\cite{lasserre} can be used to produce inner approximations of $\mathfrak R_{\sec\geq0}(n)$ as follows.

\begin{definition}
Given $n\geq2$ and a nonnegative integer $m\geq0$, consider the linear map 
\begin{equation*}
\varphi_m\colon \R[\Gr(n)]_2\longrightarrow \R[\Gr(n)]_{2m+2},\quad \varphi_m(P)= r^m\cdot P,
\end{equation*}
where $r=\sum_{i<j} x_{ij}^2\in\R[\Gr(n)]_{2}$ is the sum of squares of Pl\"ucker coordinates. Let $\mathfrak{I}_m$ be the preimage of the subset of sums of squares in $\R[\Gr(n)]_{2m+2}$ under $\varphi_m$.
\end{definition}

Clearly, every element of $\mathfrak{I}_m$ is a curvature operator $R$ with $\sec_R\geq0$. The next result shows that, conversely, every curvature operator with $\sec_R>0$ is in some $\mathfrak{I}_m$. Observe that $\mathfrak{I}_0$ is precisely the set of curvature operators with strongly nonnegative curvature, see Example~\ref{ex:grasspsdsos}.

\begin{proposition}\label{prop:innerap}
For each $n\geq2$, the collection $\mathfrak{I}_m$, $m\geq0$, is a nested sequence of $\O(n)$-invariant spectrahedral shadows such that 
\begin{equation*}
\mathfrak{R}_{\sec>0}(n)\subset \bigcup_{m\geq0}\mathfrak{I}_m \subset \mathfrak{R}_{\sec\geq0}(n).
\end{equation*} 
In particular, we have $\overline{\bigcup_{m\geq0}\mathfrak{I}_m} = \mathfrak{R}_{\sec\geq0}(n)$.
\end{proposition}

\begin{proof}
First, observe that the subset of $\R[\Gr(n)]_{2m+2}$ consisting of sums of squares is a spectrahedral shadow: it is the image of the cone of positive-semidefinite matrices under the linear map sending a symmetric matrix $A$ to
\begin{equation*}
[x]_{m+1}^{\mathrm t} A [x]_{m+1}\in\R[\Gr(n)]_{2m+2},
\end{equation*}
where $[x]_{m+1}$ denotes the column vector of all monomials of degree $m+1$.
Thus, its preimage  $\mathfrak{I}_m$ under the linear map $\varphi_m$ is a spectrahedral shadow.
Furthermore, it is $\O(n)$-invariant because $r$ is fixed by the $\O(n)$-action.
Since the product of two sums of squares is again a sum of squares, the sequence $\mathfrak{I}_m$ is nested.

Let $P\in\R[\Gr(n)]_2$ be a quadratic form that is positive on every real point of $\Gr(n)$, i.e., an element of $\mathfrak{R}_{\sec>0}(n)$. We claim that $r^m\cdot P$ is a sum of squares of elements from $\R[\Gr(n)]_{m+1}$ for all sufficiently large $m$, i.e., $P\in\mathfrak{I}_m$. This 
follows from an appropriate \emph{Positivstellensatz}, namely \cite[Cor.~4.2]{positivstellensatz} applied to the pull-back $\mathcal{L}$ of the dual of the tautological line bundle on projective space via the Pl\"ucker embedding. Note that $\R[\Gr(n)]_{k}$ is the space of global sections of $\mathcal{L}^{\otimes k}$. The last statement then follows from Remark \ref{rem:obvious} (i).
\end{proof}

\begin{example}\label{ex:zoltek}
The curvature operator $R_{\rm Zol}$ of Zoltek \cite{Zoltek79}  
in Example~\ref{ex:grasspsdsos}
has $\sec\geq0$ but does not have strongly nonnegative curvature, i.e., lies in $\mathfrak{R}_{\sec\geq0}(5)\setminus\mathfrak{I}_0$.
It can be checked that $R\in\mathfrak{I}_1$, using the package {\tt SOS}~\cite{SOSSource} for the computer algebra system {\tt Macaulay2}~\cite{M2}.
\end{example}

\begin{remark}
For general $n$, both inclusions in Proposition~\ref{prop:innerap} are \emph{strict}. For example, since $\mathfrak{I}_0=\mathfrak{R}_{\sec\geq0}(4)$, one has $\mathfrak{R}_{\sec>0}(4)\subsetneq \bigcup_{m\geq0}\mathfrak{I}_m$ in this case. On the other hand, if we interpret $R_{\rm Zol}$ from Example~\ref{ex:grasspsdsos} as an element of $\R[\Gr(7)]_2$ instead of $\R[\Gr(5)]_2$, then $R_{\rm Zol}$ is not contained in any $\mathfrak{I}_m$, and thus $\bigcup_{m\geq0}\mathfrak{I}_m\subsetneq\mathfrak{R}_{\sec\geq0}(7).$ This follows from the fact that $R_{\rm Zol}$ has a \emph{bad point} in the sense of \cite{badpoints} at the point of $\Gr(7)$ corresponding to the $2$-plane spanned by $e_6$ and $e_7$. 
\end{remark}

\begin{remark}\label{rem:preserved}
The curvature conditions corresponding to $\mathfrak I_m$, $m\geq0$, are \emph{preserved under Riemannian submersions}. More precisely, if $\pi\colon (\overline M,\overline \g)\to (M,\g)$ is a Riemannian submersion and $R_{\overline M}\in\mathfrak I_m$, then also $R_{M}\in\mathfrak I_m$. This is a direct consequence of the Gray--O'Neill formula as presented in \cite[Thm.~B]{strongpos}, which states:
\begin{equation*}
R_M=(R_{\overline M})|_{\wedge^2 TM}+3\alpha-3\mathfrak b(\alpha),
\end{equation*}
where $TM\subset T\overline M$ is the horizontal space, $\alpha$ is a quadratic form defined by the Gray--O'Neill $A$-tensor, and $\mathfrak b\colon\Sym^2(\wedge^2 TM)\to\wedge^4 TM$ is the orthogonal projection. Thus, if $\varphi_m(R_{\overline M})\in\R[\Gr(\overline n)]_{2m+2}$ is a sum of squares, then so is $\varphi_m(R_M)\in\R[\Gr(n)]_{2m+2}$, since 
$\varphi_m(3\alpha)$ is also a sum of squares and $3\mathfrak b(\alpha)$ is in the vanishing ideal of $\Gr(n)$.
\end{remark}

\subsection{Outer relaxations}
We now observe that outer approximations of $\mathfrak R_{\sec\geq0}(n)$ by spectrahedra can be constructed using \cite[Thm.~A]{weitzenbock}.
As usual, we identify traceless symmetric $p$-tensors $\psi\in\Sym^p_0\R^n$ with harmonic homogeneous polynomials $\psi\in\R[x_1,\dots,x_n]_p\cap\ker\Delta$. The curvature term induced by $R\in\Sym^2_b(\wedge^2\R^n)$ in the Weitzenb\"ock formula for traceless symmetric $p$-tensors is the
symmetric linear endomorphism $\mathcal K\big(R,\Sym^p_0\R^n\big)\colon\Sym^p_0\R^n\to\Sym^p_0\R^n$
determined by
\begin{equation}\label{eq:integralformula}
\big\langle\mathcal K\big(R,\Sym^p_0\R^n\big)\psi,\psi\big\rangle=c_{p,n}\int_{S^{n-1}} \big\langle R\big(x\wedge\nabla\psi(x)\big), x\wedge\nabla\psi(x)\big\rangle\,\dd x,
\end{equation}
where $c_{p,n}>0$ is a constant, see \cite[Prop.~3.1]{weitzenbock}.

\begin{definition}
Given $n\geq2$ and a nonnegative integer $m\geq0$, let 
\begin{equation*}
\mathfrak O_m=\big\{R\in\Sym^2_b(\wedge^2\R^n) : \mathcal K\big(R,\Sym^p_0\R^n\big)\succeq0 \text{ for all } 1\leq p\leq m+1 \big\}.
\end{equation*}
\end{definition}

\begin{proposition}\label{prop:weitz}
For each $n\geq2$, the collection $\mathfrak O_m$, $m\geq0$, is a nested sequence of $\O(n)$-invariant spectrahedra such that 
\begin{equation*}
\bigcap_{m\geq0}\mathfrak O_m=\mathfrak R_{\sec\geq0}(n).
\end{equation*}
\end{proposition}

\begin{proof}
Since $\mathcal K\big(R,\Sym^p_0\R^n\big)$ depends linearly on $R$ and is $\O(n)$-equivariant,
 $\mathfrak O_m$ are (finite) intersections of $\O(n)$-invariant spectrahedra, hence $\O(n)$-invariant spectrahedra themselves.
The inclusion $\mathfrak R_{\sec\geq0}(n)\subset\bigcap_{m\geq0}\mathfrak O_m$ holds since
the integrand in \eqref{eq:integralformula} is a sectional curvature;  the reverse inclusion follows by \cite[Thm.~A]{weitzenbock}.
\end{proof}

\begin{remark}
Note that the first step $\mathfrak O_0$ is precisely the set of curvature operators with nonnegative Ricci curvature~\cite[Ex.~2.2]{weitzenbock}. Thus, in contrast with the inner approximations (see Remark~\ref{rem:preserved}),
these curvature conditions are in general \emph{not preserved} under Riemannian submersions, see e.g.~\cite{pro-wilhelm}.
\end{remark}

\begin{proof}[Proof of Theorem~\ref{mainthm:approx}]
Without loss of generality, we may consider only the case $k=0$, see Remark~\ref{rem:obvious}.
The result now follows from Propositions~\ref{prop:innerap} and \ref{prop:weitz}.
\end{proof}

\subsection{\texorpdfstring{Algorithms to detect $\sec\geq0$}{Algorithms to detect nonnegative sectional curvature}}
The relaxations constructed above enable the use of \emph{semidefinite programming} to test membership in $\mathfrak R_{\sec\geq0}(n)$.
A \emph{semidefinite program} optimizes a linear functional over a spectrahedral shadow $S$; in particular, it can be used to test whether $S$ is empty, and whether a given point belongs to $S$. Using interior-point methods, one can solve a semidefinite program up to a fixed precision in polynomial time in the size of the program description, see e.g.~\cite{nesterov}.

\begin{algorithm}[ht]
\caption{Query $\sec_R\geq0$ via iterated semidefinite programs, for $n\geq5$}  
\label{alg:relax}
\SetAlgoLined
\DontPrintSemicolon

\SetKwInOut{Input}{input}\SetKwInOut{Output}{output}
\SetKw{Def}{def}
\SetKw{Var}{var}

\Input{$R\in\Sym^2_b(\wedge^2\R^n)$}
\Output{TRUE if $\sec_R\geq0$, FALSE otherwise}
\Var $m:=0$\;
\Var finished := FALSE\;

\While{\rm finished = FALSE}{
\If{$R\in\mathfrak{I}_m$ \tcc{Semidefinite Programming used here} }{finished := TRUE\; \KwRet{\rm TRUE}}
\If{$R\not\in\mathfrak{O}_m$ \tcc{Semidefinite Programming used here} }{finished := TRUE\; \KwRet{\rm FALSE}}
$m := m+1$\;
}\end{algorithm}

\begin{proposition}
For all $R\in\Sym^2_b(\wedge^2\R^n)\setminus\mathcal B$, where $\mathcal B:=\mathfrak{R}_{\sec\geq0}(n)\setminus\bigcup_{m\geq0}\mathfrak{I}_m$, Algorithm \ref{alg:relax} terminates and returns {\rm TRUE} if and only if $\sec_R\geq0$. The set of bad inputs $\mathcal B$ has measure zero in $\Sym^2_b(\wedge^2\R^n)$ and is contained in $\partial\mathfrak{R}_{\sec\geq0}(n)$.
\end{proposition}

\begin{proof}
These claims follow immediately from Propositions \ref{prop:innerap} and \ref{prop:weitz}.
\end{proof}

\subsection{\texorpdfstring{Efficient algorithms for $n=4$}{Efficient algorithms in dimension 4}}\label{subsec:algorithms}
Although semidefinite programming would not require iterations to detect membership in the spectrahedral shadow $\mathfrak R_{\sec\geq0}(4)$, more efficient algorithms follow from the Finsler--Thorpe trick (Proposition~\ref{prop:thorpe}).

Recall that the classical \emph{Sturm's algorithm} returns the number of real roots (disregarding multiplicities) of a given univariate real polynomial $p\in\R[x]$ in any interval $[a,b]$ with $-\infty\leq a<b\leq +\infty$ and $p(a)\neq0$ and $p(b)\neq0$, see e.g.~\cite[Cor.~1.2.10]{bookBCR}, or \cite[Sec.~2.2.2]{BPR} for a more algorithmic viewpoint.
This method produces  partitions $-\infty=a_1<a_2<\dots<a_N=+\infty$ which are \emph{root-isolating}, that is, $p(a_j)\neq0$ for all $j$ and $p(x)$ has exactly one root in $[a_j,a_{j+1}]$.
In what follows, we convention that $p(\pm\infty)$ are  interpreted as limits.
Combining this procedure with Euclid's division algorithm (to compute greatest common divisors of polynomials), one can produce a \emph{common root-isolating partition} for any finite collection of polynomials $p_i\in\R[x]$, i.e., 
$-\infty=a_1<a_2<\dots<a_N=+\infty$
such that $p_i(a_j)\neq0$ for all $i$ and $j$, and $[a_j,a_{j+1}]$ contains exactly one root of some $p_i(x)$. Note that if more than one $p_i(x)$ vanishes in $[a_j,a_{j+1}]$, then they must do so at the same point.

\begin{algorithm}[ht]
\caption{Query $\sec_R>0$ in dimension $n=4$}\label{alg:poscurv}
\SetAlgoLined
\DontPrintSemicolon

\SetKwInOut{Input}{input}\SetKwInOut{Output}{output}
\SetKw{Def}{def}

\Input{$R\in\Sym^2_b(\wedge^2\R^4)$}
\Output{TRUE if $\sec_R>0$, FALSE otherwise}
\Def $\sigma_i(x)\in\R[x]$, $1\leq i\leq 6$, such that
$\det(R+x\,*-\lambda\id)=\lambda^6+\displaystyle\sum_{i=1}^6 (-1)^i\sigma_i(x)\lambda^{6-i}$\; \tcc{i.e., $\sigma_i(x)$ are the elementary symmetric polynomials on eigenvalues of $R+x\,*$} 
\Def $-\infty=a_1<a_2<\dots<a_N=+\infty$ common root-isolating partition for $\sigma_i(x)\in\R[x]$, $1\leq i\leq 6$ \tcc{Sturm's Algorithm used here} 
\eIf{\rm $\exists j$ such that $\sigma_i(a_j)>0$ for all $1\leq i\leq 6$}{\KwRet{\rm TRUE}}{\KwRet{\rm FALSE}}
\end{algorithm}

\begin{proposition}\label{prop:algorithmsecpos}
For all $R\in\Sym^2_b(\wedge^2\R^4)$, Algorithm~\ref{alg:poscurv} terminates and returns {\rm TRUE} if and only if $\sec_R>0$.
\end{proposition}

\begin{proof}
By the Finsler--Thorpe trick (Proposition~\ref{prop:thorpe}), $\sec_R>0$ if and only if there exists $x\in\R$ such that $R+x\*\succ0$, that is, the elementary symmetric polynomials $\sigma_i(x)$, $1\leq i\leq6$, in the eigenvalues of $R+x\,*$ are all positive. 

Suppose the algorithm returns TRUE. Then there exists $1\leq j\leq N$ such that
$\sigma_i(a_j)>0$ for all $1\leq i\leq 6$, so $R+a_j\,*\succ0$ and hence $\sec_R>0$. Conversely, if $\sec_R>0$, let $x_0\in\R$ be such that $R+x_0\,*\succ0$. Set $1\leq j\leq N$ such that $x_0\in [a_j,a_{j+1}]$. If $x_0=a_j$ or $x_0=a_{j+1}$, then the algorithm clearly returns TRUE, so we may assume $x_0\in (a_j,a_{j+1})$. Since  $-\infty=a_1<a_2<\dots<a_N=+\infty$ is a common root-isolating partition, one of the intervals $(a_j,x_0)$ or $(x_0,a_{j+1})$ contains no roots of any $\sigma_i(x)$. Thus, either $\sigma_i(a_j)>0$ for all $1\leq i \leq 6$,
or $\sigma_i(a_{j+1})>0$ for all $1\leq i \leq 6$; which implies that the algorithm returns TRUE.
\end{proof}

\begin{algorithm}[ht]
\caption{Query $\sec_R\geq0$ in dimension $n=4$}\label{alg:nonnegcurv}
\SetAlgoLined
\DontPrintSemicolon

\SetKwInOut{Input}{input}\SetKwInOut{Output}{output}
\SetKw{Def}{def}
\SetKw{Var}{var}

\Input{$R\in\Sym^2_b(\wedge^2\R^4)$}
\Output{TRUE if $\sec_R\geq0$, FALSE otherwise}
\Def $\sigma_i(x)\in\R[x]$, $1\leq i\leq 6$, such that
$\det(R+x\,*-\lambda\id)=\lambda^6+\displaystyle\sum_{i=1}^6 (-1)^i\sigma_i(x)\lambda^{6-i}$\; \tcc{i.e., $\sigma_i(x)$ are the elementary symmetric polynomials on eigenvalues of $R+x\,*$} 
\Def $-\infty=a_1<a_2<\dots<a_N=+\infty$ common root-isolating partition for $\sigma_i(x)\in\R[x]$, $1\leq i\leq 6$ \tcc{Sturm's Algorithm used here} 
\Var answer := {\rm FALSE}\;
\For{$j=1,\dots, N-1$}{
\If{\rm $\forall 1\leq i\leq 6$ with $\sigma_i(a_j)<0$ and $\sigma_i(a_{j+1})<0$, $\sigma_i(x)$ has a root in $(a_j,a_{j+1})$\label{algline:iftest}}{answer := \rm TRUE}
}
\KwRet{\rm answer}
\end{algorithm}

\begin{proposition}
For all $R\in\Sym^2_b(\wedge^2\R^4)$, Algorithm~\ref{alg:nonnegcurv} terminates and returns {\rm TRUE} if and only if $\sec_R\geq0$.
\end{proposition}

\begin{proof}
Analogously to Proposition~\ref{prop:algorithmsecpos}, $\sec_R\geq0$ if and only if there exists $x\in\R$ such that $\sigma_i(x)\geq0$ for all $1\leq i\leq 6$.

Suppose the algorithm returns TRUE, so there exists $1\leq j\leq N-1$ such that for all $1\leq i\leq 6$ with $\sigma_i(a_j)<0$ and $\sigma_i(a_{j+1})<0$, the polynomial $\sigma_i(x)$ has a root in $(a_j,a_{j+1})$. Let $x_0\in (a_j,a_{j+1})$ be the only root of some $\sigma_i(x)$ in that interval. We claim that $\sigma_i(x_0)\geq0$ for all $1\leq i\leq 6$, hence $\sec_R\geq0$. This is shown by examining the signs of $\sigma_i(a_j)$ and $\sigma_i(a_{j+1})$
and using that $-\infty=a_1<a_2<\dots<a_N=+\infty$ is a common root-isolating partition for the $\sigma_i(x)$, as follows:
\begin{enumerate}
\item If $\sigma_i(a_j)<0$ and $\sigma_i(a_{j+1})<0$, then $\sigma_i(x_0)=0$ by the test in line {\small\bf\ref{algline:iftest}};
\item If $\sigma_i(a_j)>0$ and $\sigma_i(a_{j+1})>0$, then clearly $\sigma_i(x_0)\geq0$;
\item If $\sigma_i(a_j)$ and $\sigma_i(a_{j+1})$ have opposite signs, then $\sigma_i(x_0)=0$ by the Intermediate Value Theorem.
\end{enumerate}

Conversely, if $\sec_R\geq0$, choose $x_0\in\R$ such that $\sigma_i(x_0)\geq0$ for all $1\leq i\leq 6$. Let $1\leq j\leq N-1$ be such that $x_0\in [a_j,a_{j+1}]$, and $1\leq i\leq 6$ be such that $\sigma_i(a_j)<0$ and $\sigma_i(a_{j+1})<0$.  If $\sigma_i(x_0)>0$, then $\sigma_i(x)$ would have more than one root in $(a_j,a_{j+1})$ contradicting the common root-isolating property, so $\sigma_i(x_0)=0$. Therefore, the algorithm returns TRUE.
\end{proof}

\appendix
\section{Curvature operators of Semi-Riemannian manifolds}
\label{sec:semiriem}

Recall that a \emph{semi-Riemannian} (or \emph{pseudo-Riemannian}) manifold $(M,\g)$ is a smooth manifold endowed with a
\emph{semi-Riemannian metric} $\g$, i.e., a smooth section of the bundle $\Sym^2(TM)$ of symmetric bilinear forms that is nondegenerate (but possibly indefinite). As in the rest of the paper, since all our considerations are pointwise, given $p\in M$ we shall identify $T_pM\cong\R^n$, and consider
\begin{equation*}
\g(X,Y)= \langle GX,Y \rangle = -\textstyle\sum\limits_{i=1}^\nu x_iy_i + \sum\limits_{j=\nu+1}^n x_jy_j,
\end{equation*}
where $\langle\cdot,\cdot\rangle$ is the standard inner product, $X=(x_1,\dots,x_n)$, $Y=(y_1,\dots,y_n)$, and
\begin{equation*}
G=\diag\!\big(\underbrace{-1,\dots,-1}_\nu,\underbrace{1,\dots,1}_{n-\nu}\big).
\end{equation*}
The integer $0\leq \nu\leq n$ is called the \emph{signature} of $\g$, and we henceforth assume $0<\nu<n$, i.e., $\pm\g$ are \emph{not} Riemannian.
Exactly as in the Riemannian case, semi-Riemannian metrics $\g$ determine a (unique) Levi--Civita connection $\nabla$ on $TM$, see e.g.~\cite[p.~11]{oneill-book}, and hence a curvature operator $R$, as in \eqref{eq:curvopmanifold}, that satisfies the first Bianchi identity \eqref{eq:bianchi}. However, these equations must be interpreted appropriately: the standard inner product $\langle\cdot,\cdot\rangle$ on $\wedge^2 T_pM\cong\wedge^2 \R^n$ has to be replaced by the nondegenerate (but indefinite) symmetric bilinear form induced by $\g$, namely 
\begin{equation*}
\begin{aligned}
Q(X\wedge Y,Z\wedge W) &= \langle (G\wedge G) (X\wedge Y),Z\wedge W\rangle\\
&=\g(X,Z)\g(Y,W)-\g(X,W)\g(Y,Z),
\end{aligned}
\end{equation*}
where $(G\wedge G)(X\wedge Y)=GX\wedge GY$.
Denote by $\Sym_Q^2(\wedge^2\R^n)$ the set of $Q$-symmetric $R\colon\wedge^2\R^n\to\wedge^2\R^n$, i.e., such that, for all $X\wedge Y,Z\wedge W\in\wedge^2\R^n$,
\begin{equation*}
Q(R(X\wedge Y),Z\wedge W)=Q(X\wedge Y,R(Z\wedge W)),
\end{equation*}
and similarly for $\Sym^2_{Q,b}(\wedge^2\R^n)$ and \eqref{eq:bianchi}.
Decomposable elements $X\wedge Y\in \wedge^2 \R^n$ are called \emph{nondegenerate} if the restriction of $\g$ to the $2$-plane spanned by $X$ and $Y$ is nondegenerate, and \emph{degenerate} otherwise. Note that $X\wedge Y$ is nondegenerate if and only if $Q(X\wedge Y,X\wedge Y)\neq 0$. Furthermore, $X\wedge Y$ is called \emph{definite} 
if $Q(X\wedge Y,X\wedge Y)>0$, and \emph{indefinite} 
if $Q(X\wedge Y,X\wedge Y)<0$; as these are equivalent to the restriction of $\g$ to $\operatorname{span}(X,Y)$ being definite, and indefinite, respectively.

The sectional curvature determined by $R\in\Sym_Q^2(\wedge^2\R^n)$ is 
\begin{equation*}
\sec_{Q,R}(X\wedge Y) = \frac{Q(R(X\wedge Y),X\wedge Y)}{Q(X\wedge Y,X\wedge Y)},
\end{equation*}
and is only defined for nondegenerate $X\wedge Y$, cf.~\eqref{eq:sec}. By a well-known result of Kulkarni, see \cite[p.~229]{oneill-book}, the only $R\in\Sym^2_{Q,b}(\wedge^2\R^n)$ with $\sec_{Q,R}(X\wedge Y)\geq 0$ for all nondegenerate $X\wedge Y$ is $R=0$. Furthermore, if the restriction of $\sec_{Q,R}$ to \emph{either} definite \emph{or} indefinite elements is bounded (from above and below), then $\sec_{Q,R}$ is constant everywhere. Thus, Definition~\ref{def:rseckn} becomes vastly uninteresting.

Nevertheless, a suitable generalization of Definition~\ref{def:rseckn} to the case of indefinite signature is given by the set $\mathfrak R_{\sec\geq k}(n,\nu)$ of $R\in\Sym^2_{Q,b}(\wedge^2\R^n)$ such that
\begin{equation}\label{eq:Rgeqk-indef}
Q(R(X\wedge Y),X\wedge Y) \geq k\, Q(X\wedge Y,X\wedge Y), \quad \text{ for all } X\wedge Y\in\wedge^2\R^n.
\end{equation}
Note that \eqref{eq:Rgeqk-indef} is equivalent to $\sec_{Q,R}\geq k$ on definite elements and $\sec_{Q,R}\leq k$ on indefinite elements. It is easy to see that $\mathfrak R_{\sec\geq k}(n,\nu)=\mathfrak R_{\sec\geq k}(n)$ if $\nu=0$ or $\nu=n$, and $\mathfrak R_{\sec\geq k}(n,n-\nu)=\mathfrak R_{\sec\geq k}(n,\nu)$.

Arguably, condition \eqref{eq:Rgeqk-indef} is not only \emph{algebraically,} but also \emph{geometrically} more natural than $\sec_{Q,R}\geq k$, since semi-Riemannian manifolds satisfy it pointwise if and only if they satisfy a local Alexandrov triangle comparison on the signed lengths of geodesics~\cite{alexander-bishop}.
This curvature condition was first considered in \cite{andersson-howard}, where comparison results for the Riccati equation are proven, and is also related to space-time convex functions~\cite{alex-karr}.
Our Theorem~\ref{mainthm:spectrahedra} carries over verbatim to this context:

\begin{mainthmp}\label{mainthm:spectrahedra2}
For all $k\in\R$ and $0\leq \nu\leq n$, the set $\mathfrak R_{\sec\geq k}(n,\nu)$ is:
\begin{enumerate}[\indent \rm (1)]
\item  not a spectrahedral shadow, if $n\geq5$;
\item  a spectrahedral shadow, but not a spectrahedron, if $n=4$;
\item  a spectrahedron, if $n\leq3$.
\end{enumerate}
\end{mainthmp}

\begin{proof}
Consider the linear isomorphism $\psi_Q\colon \Sym^2_Q(\wedge^2\R^n)\to\Sym^2(\wedge^2\R^n)$, given by $\psi_Q(R)=(G\wedge G)\cdot R$. Note that it restricts to a linear isomorphism
\begin{equation*}
\psi_Q\colon \Sym^2_{Q,b}(\wedge^2\R^n)\longrightarrow\Sym^2_b(\wedge^2\R^n) \cong \R[\Gr(n)]_2,
\end{equation*}
and $\mathfrak R_{\sec\geq k}(n,\nu)=\psi_Q^{-1}\big(\mathfrak R_{\sec\geq k}(n)\big)$, so all conclusions follow from Theorem~\ref{mainthm:spectrahedra}.
\end{proof}

Moreover, preimages of the inner and outer approximations in Theorem~\ref{mainthm:approx} by $\psi_Q$ give analogous approximations of $\mathfrak R_{\sec\geq k}(n,\nu)$. However, these are no longer $\O(n)$-invariant, since $\psi_Q$ is not $\O(n)$-equivariant unless $\nu=0$ or $\nu=n$. Finally, Theorem~\ref{mainthm:4dim} also carries over to $\mathfrak R_{\sec\geq k}(4,\nu)$, precomposing \eqref{eq:pk} with $\psi_Q$.

\section{Irreducibility of the discriminant of symmetric matrices}\label{appendix}

In this Appendix, we study irreducibility of discriminants of symmetric matrices. Although the techniques are standard, 
we give complete proofs for the convenience of the reader, as the following does not seem to be easily available in the literature:

\begin{proposition}\label{prop:discirred}
The discriminant of symmetric matrices $\disc\colon\Sym^2(\C^n)\to\C$ is irreducible over $\C$ for all $n\geq3$.
\end{proposition}

\begin{remark}
The polynomial $\disc\colon\Sym^2(\C^n)\to\C$ is a constant if $n=1$, and is \emph{not irreducible} if $n=2$ since it is a product of two complex conjugate linear forms.
\end{remark}

Consider the conjugation action on $\Sym^2(\C^n)$ of the (complex) Lie groups:
\begin{equation*}
\begin{aligned}
\O(n,\C)&=\big\{S\in\GL(n,\C):\, S^\mathrm t S=\id\big\},\\
\SO(n,\C)&=\big\{S\in\O(n,\C):\, \det(S)=1\big\}.
\end{aligned}
\end{equation*}
Recall that $\dim_\C \O(n,\C)=\dim_\C \SO(n,\C)=\binom{n}{2}$, and that $\SO(n,\C)$ is an irreducible affine variety. 

\begin{lemma}\label{lem:so}
Let $A\in\Sym^2(\C^n)$ be a symmetric matrix whose $i${\rm th} row has only zero entries, except possibly for its $i${\rm th} entry. Then the $\O(n,\C)$-orbit of $A$ coincides with its $\SO(n,\C)$-orbit.
\end{lemma}

\begin{proof}
Let $S\in\O(n,\C)\setminus\SO(n,\C)$ be the reflection on the $i${th} coordinate, i.e., the diagonal matrix with entries $1$ on the diagonal, except for a $-1$ at the $i${th} position. Clearly, $S$ is in the isotropy of $A$ and $\O(n,\C)=\SO(n,\C)\cup (\SO(n,\C)\cdot S)$.
\end{proof}

In the following, we describe the $\O(n,\C)$-orbit of a symmetric matrix $A\in\Sym^2(\C^n)$ by expressing it in terms of a convenient canonical form.
To this end, following the notation of \cite{pencils}, let $\Delta_k,\Lambda_k\in\Sym^2(\C^k)$ be the matrices given by
\begin{equation*}
\Delta_k=\begin{pmatrix}
&&&&1\\&&&1&\\&&\iddots&&\\&1&&&\\1&&&&
\end{pmatrix},
\qquad
\Lambda_k=\begin{pmatrix} &&&&0\\&&&0&1\\&&\iddots&1&\\&\iddots&\iddots&&\\0&1&&&
\end{pmatrix},
\end{equation*}
and fix $R_k\in\GL(k,\C)$ such that $R_k \Delta_k R_k^\mathrm t=\id_k$. Note that $\Lambda_1=0$. Furthermore, given $\lambda\in\C$, define $M_k(\lambda)\in\Sym^2(\C^k)$ by
\begin{equation*}
M_k(\lambda)=\lambda \id_k+R_k \Lambda_k R_k^\mathrm t.
\end{equation*}

\begin{lemma}\label{lem:orbits}
Given $A\in\Sym^2(\C^n)$, there exist $\lambda_1,\dots,\lambda_\ell\in\C$ eigenvalues of $A$, and $k_1,\dots,k_\ell\in\mathds N$, so that the $\O(n,\C)$-orbit of $A$ contains the block diagonal matrix
\begin{equation}\label{eq:canform}
\begin{pmatrix}
M_{k_1}(\lambda_1) && \\&\ddots & \\ && M_{k_\ell}(\lambda_\ell) \end{pmatrix}.
\end{equation}
Moreover, the characteristic polynomial of $A$ is $\det(A-t\id)=\prod_{i=1}^\ell (\lambda_i-t)^{k_i}$.
\end{lemma}

\begin{proof}
We make use of the fact that one can bring pencils of symmetric matrices over $\C$ to a certain standard form. As proved e.g.~in \cite[Sec.~5]{pencils}, there 
exist $\lambda_1,\dots,\lambda_\ell\in\C$ eigenvalues of $A$, $k_1,\dots,k_\ell\in\mathds N$,
and $S\in\GL(n,\C)$ with the following property: for all $\rho\in\C$, the matrix $S(\rho\id_n+A) S^\mathrm t$ is block diagonal with $\ell$ blocks of the form $(\rho+\lambda_i) \Delta_{k_i}+\Lambda_{k_i}$.

Letting $R\in\GL(n,\C)$ be the block diagonal matrix with blocks $R_{k_1},\ldots,R_{k_\ell}$, we have that $RS(\rho\id_n+A) S^\mathrm tR^\mathrm t$ is a block diagonal matrix with blocks $M_{k_i}(\rho+\lambda_i)$.
Since this holds for all $\rho\in\C$, it follows that $RS\in\O(n,\C)$ and that 
$RS\, A\, (RS)^\mathrm t$ has the desired form \eqref{eq:canform}. 

The characteristic polynomial of $A$ is equal to that of 
\eqref{eq:canform}, which is the product of the characteristic polynomials of its blocks $M_{k_i}(\lambda_i)$. These can be computed as:
\begin{equation*}
\begin{aligned}
\det(M_k(\lambda)-t\id_k)&=\det(R_k(\Lambda_k+(\lambda-t)\Delta_k)R_k^\mathrm t)\\
&=\det(R_k)^2\det(\Lambda_k+(\lambda-t)\Delta_k)\\
&=(\lambda-t)^k,
\end{aligned}
\end{equation*}
because $\det(R_k)^2=(-1)^{\lfloor n/2\rfloor}$ and $\det(\Lambda_k+(\lambda-t)\Delta_k)=(-1)^{\lfloor n/2\rfloor}(\lambda-t)^k$.
\end{proof}

We are now in the position to prove the main result of this Appendix:

\begin{proof}[Proof of Proposition~\ref{prop:discirred}]
The zero set $V$ of $\disc\colon\Sym^2(\C^n)\to\C$ is a hypersurface in $\Sym^2(\C^n)$, and hence an equidimensional variety of (pure) codimension $1$. 
Furthermore, since $\disc\colon\Sym^2(\C^n)\to\C$ is $\O(n,\C)$-invariant, by Lemma~\ref{lem:orbits} we have that $V$ is the union of the $\O(n,\C)$-orbits of block diagonal matrices with blocks $M_{k_1}(\lambda_1),\dots,M_{k_\ell}(\lambda_\ell)$ for all $\lambda_i\in\C$, and $k_1+\ldots+k_\ell=n$ with at least one $k_i\geq2$, that is, $\ell\leq n-1$.
For fixed $k_1,\ldots,k_\ell$, the set of such block matrices is parametrized by $\C^\ell$, and hence the union of the $\O(n,\C)$-orbits of such matrices has dimension $\leq \ell+\dim_\C \O(n,\C)$. If $\ell<n-1$, then $\ell+\dim_\C \O(n,\C)<\dim_\C V$.
Thus, since $V$ is equidimensional, it is the closure of the union of the $\O(n,\C)$-orbits of all matrices
\begin{equation*}
M(\lambda)=\begin{pmatrix} M_2(\lambda_1)&&&\\&\lambda_2&&\\&&\ddots&\\
&&&\lambda_{n-1}\end{pmatrix}
\end{equation*}
with $\lambda=(\lambda_1,\ldots,\lambda_{n-1})\in\C^{n-1}$. Since $n\geq3$, it suffices to take $\SO(n,\C)$-orbits by Lemma \ref{lem:so}. In other words, $V$ is the closure of the image of the map
\begin{equation*}
\SO(n,\C)\times\C^{n-1}\longrightarrow\Sym^2(\C^n), \quad (S,\lambda)\mapsto S\, M(\lambda)\, S^\mathrm t.
\end{equation*}
As the source is irreducible, the (closure of the) image is irreducible as well, which shows that $V$ is irreducible. Therefore, $\disc=\phi^m$ for some irreducible polynomial $\phi\colon\Sym^2(\C^n)\to\C$ and $m\in\mathds N$, so it remains to show that $m=1$. This can be seen, e.g., considering the restriction of $\disc\colon\Sym^2(\C^n)\to\C$ to the curve $$\begin{pmatrix}1+x&\sqrt{-1}&&&\\ \sqrt{-1}&-1-x&&&\\&&1&&\\&&&\ddots&\\&&&&n-2\end{pmatrix}\in \Sym^2(\C^n),$$ which is a univariate polynomial in $x$ with a \emph{simple} root at $0$, hence $m=1$.
\end{proof}

\def\cprime{$'$}

\end{document}